\documentclass[10pt]{amsart}
\usepackage{amsmath, amssymb, amsthm,  epsfig}
\usepackage{amsfonts}
\theoremstyle{plain}
\newtheorem{theorem}{Theorem}[section]
\newtheorem*{thrm*}{Theorem}
\newtheorem{lemma}[theorem]{Lemma}
\newtheorem{proposition}[theorem]{Proposition}
\newtheorem{corollary}[theorem]{Corollary}
\newtheorem{definition}[theorem]{Definition}
\newtheorem{remark}[theorem]{Remark}
\numberwithin{equation}{section}

\def \a {{\alpha}}

\def \G {{\Gamma}}
\def \s {{\sigma}}

\def \R {{\mathbb {R}}}
\def \H {{\mathbb {H}}}

\def \N {{\mathbb {N}}}

\def \e {{\varepsilon}}

\def \r {{\varrho}}
\def \u {{\bar{u}}}

\def \x {{\xi}}

\def \O {{\Omega}}
\def \phi {{\varphi}}

\def \tilde {\widetilde}

\def\p{\partial}

\begin{document}
\title[minimal intrinsic graphs in Heisenberg groups]
 {Smoothness  of  Lipschitz  minimal intrinsic graphs     \\
   in Heisenberg groups $\H^n$, $n>1$}

\author{Luca Capogna}\address{Department of Mathematics,
University of Arkansas, Fayetteville, AR 72701}\email{lcapogna@uark.edu}
\author{Giovanna Citti}\address{Dipartimento di Matematica, Piazza Porta S. Donato 5,
40126 Bologna, Italy}\email{citti@dm.unibo.it}
\author{Maria Manfredini}\address{Dipartimento di Matematica, Piazza Porta S. Donato 5,
40126 Bologna, Italy}\email{manfredi@dm.unibo.it}
\keywords{regularity of solutions of PDE, minimal surfaces, sub-Riemannian geometry, Heisenberg group\\
The authors are partially funded by NSF Career grant DMS-0124318
(LC) and by INDAM
(GC) and (MM)}

\begin{abstract} We prove that Lipschitz  intrinsic graphs in the Heisenberg groups $\H^n$, with $n>1$,
which are  vanishing viscosity solutions of the minimal surface equation are smooth.
\end{abstract}
\thanks{ }

\maketitle


\section{Introduction}
The Heisenberg group is a Lie group with Lie algebra $\R^{2n+1}$ endowed  with a stratification
 $V_1 \oplus V_2,$ where $V_1$ has dimension $2n$, and
$V_2=[V_1, V_1]$ has dimension $1$. Since we are interested in
non-characteristic graphs, it is convenient that we use canonical coordinates of the second kind (the so called {\it polarized coordinates } \cite{cdpt}) and  denote $(s, x)$ the elements of the group, where
 $x=(x_1,...,x_{2n})$. Accordingly we will choose a basis of
the Horizontal tangent space $V_1$ as follows:
\begin {equation}\label{defcampi}
\begin{split}
&X_s = \p_s,  X_i=\p_{i}, \text{ for }i=1,...,n-1, \\
&X_{i}=\p_{ i}-x_{i-n+1}\p_{2n}, \,\,\text{ for
}i=n,...,2n-1.
\end{split}
\end {equation}
 This set of vectors can be completed to be a   basis of the tangent space by 
 adding the
vector
$$\p_{2n}\in V_2.$$ 

The notion of intrinsic regular surface has been studied in
\cite{FSSC1},  \cite{CM1}. Such a surface is  the graph of a function
$u:\R^{2n}\rightarrow \R$, and can be represented as
$$M=\{(s, x): \  s= u(x)\}.$$
Note that $C^1$  intrinsic graphs are always non-characteristic\footnote{ i.e. $T_pM\not = span\{X_s,X_1,...,X_{2n-1}\}(p)$, for all $p\in M$}.
According to a version of the implicit function theorem (\cite{FSSC1} and  \cite{CM1}), any
level surface $\{f(s,x)=c\}\subset \H^n$ of  functions $f:\H^n\to\R$ with continuous derivatives along the directions (\ref{defcampi}), can locally (near non-characteristic points) be expressed
as  an intrinsic graph of a function $u:\Omega\to \R$, $\Omega\subset \R^{2n}$. Moreover the
$C^1_H$ smoothness of $f$ implies that the  function $u$ is regular with respect to the
projection on its domain of the vector fields in (\ref{defcampi})
(see \cite{CM1} and \cite{ASV}).
 Since $X_s$ has null projection
of the domain of $u$, the regularity of this function will be
described in terms of the vector fields:
\begin {equation}\label{nonlincampi}  X_{i, u}=X_{i}\text{ for } i\leq 2n-2,
  X_{2n-1, u}=\p_{2n-1}+ u(x)\p_{2n},\end{equation} In
particular $X_{2n-1, u}$ is a non linear vector field, since it
depends on $u$. Note that the
 vector fields $X_{1, u}, ...,X_{2n-1, u},$
satisfy H\"ormander's finite rank condition in $\R^{2n}$. Consequently they give rise
to a control distance $d_u$, whose metric balls $B_u(x,r)$ have volume comparable
to $r^Q$, with $Q=2n+1$ the homogenous dimension of the space $(\R^{2n},d_u)$.

\bigskip
The notion of mean curvature has been recently introduced  as the
first variation\footnote{For variations which do not move the characteristic set} of the area functional. Several first variation formula have been independently established in receent years, see
for instance   \cite{dgn:minimal}, \cite{chy}, \cite{chmy:minimal}, 
 \cite{BC}, \cite{RR}, \cite{RR2}, \cite{pau:cmc-carnot},
 \cite{montefalcone}, \cite{selby},  \cite{Sherbakova}.
 For an introduction to the sub-Riemannian 
geometry  of the Heisenberg group  and a more detailed list of references see \cite{cdpt}. 
 The prescribed  mean curvature   equation for
intrinsic graphs (over $\Omega  \subset \R^{2n}$) in the Heisenberg groups of dimension $n>1$  has
the following expression
\begin{equation}\label{MAINgeq2} Lu =  \sum_{i=1}^{2n-1}  X_{i,  u}\Bigg(\frac{X_{i,  u} u}{\sqrt{1 +
|\nabla_{ u}u|^2}}\Bigg)=f, \text{ for }x\in \Omega \subset \R^{2n}.
\end{equation}
where
$$\nabla_u = (X_{1, u}, \ldots, X_{2n-1, u}).$$ If $u\in C^2(\Omega)$ is
a solution of \eqref{MAINgeq2} for $f=0$ then its graph is a critical point 
of the perimeter and consequently it is called a minimal intrinsic graph.

Properties of regular minimal surfaces have been investigated in
 \cite{gn:isoperimetric}, \cite{pau:minimal},
\cite{chmy:minimal}, \cite{chy}, \cite{gp:bernstein}, \cite{dgn2},
\cite{bascv} and \cite{ni:cmc}. 

Since minimal surfaces arise as critical points of the
perimenter functional, the variational formulation naturally provides 
several notions of   non regular
solutions (see for instance  \cite{gn:isoperimetric}, \cite{pau:minimal} and  \cite{chy}). Indeed existence of  BV minimizers  of the perimeter  is proved in 
 \cite{gn:isoperimetric}, \cite{pau:minimal} using direct methods of the calculus of variations, 
More recently,  existence of  Lipschitz continuous vanishing viscosity solutions has  been  studied in  
\cite{chy}. Such solutions arise as  the sub-Riemannian mean curvature
equation is approximated by Riemannian problems which express the
mean curvature in an approximating Riemannian metrics (see
\cite{pau:minimal} and \cite{chy} for the relation between Riemannian and
sub-Riemannian curvature). The Riemannian approximation of \eqref{MAINgeq2}  is
\begin{equation}\label{MAINgeq2e} L_\e u = \sum_{i=1}^{2n}
X^\e_{i,  u}\Bigg(\frac{X^\e_{i, u}u}{\sqrt{1 + |\nabla^\e_{
u}u|^2}}\Bigg)=f, \text{ for }x\in \Omega  \subset \R^{2n}.
\end{equation}
where \begin{equation}\label{ecampi}X^\e_{i, u} = X_{i, u}
\,\text{ for } \,i\leq 2n-1, \quad X^\e _{2n ,u} = \e\p_{2n} \quad
\text{and} \quad \nabla^\e_u = (X^\e_{1, u}, \ldots, X^\e_{2n,
u}).\end{equation}

\begin{definition}\label{vvs}
Letting  $C^1_E$ denote the standard Euclidean $C^1$ norm, we will say
that an Euclidean  Lipschitz continuous function $u$ is a
{\it vanishing viscosity solution } of \eqref{MAINgeq2} in an open set
$\Omega$, if there exists a sequence $u_\e$ of
smooth
 solutions of
\eqref{MAINgeq2e} in $\Omega$ such that for every compact set
$K\subset \Omega$
\begin{itemize}
  \item $||u_{\e}||_{C^1_E(K)} \leq C$ for every $\e$;
  \item $u_\e\rightarrow u$ as $\e\rightarrow 0$ pointwise  a.e.
  in $\Omega$.
\end{itemize}
\end{definition}

As mentioned above, existence of this type of viscosity solutions in the case of $t-$graphs, i.e.
graphs of the form $x_{2n}=g(s,x_1,...,x_{2n-1})$, has been proved in
\cite[Theorem A and Theorem 4.5]{chy}. For such graph the corresponding PDE is
more degenerate than \eqref{MAINgeq2e} as characteristic points are allowed
(indeed, much of the analysis in \cite{chy} and \cite{chmy:minimal} is focused
on the study of solutions near such points).  In the same paper the authors prove that such solutions are
minimizers of the perimeter and address questions of uniqueness and comparison
theorems as well.
%
%
The problem of regularity  of minimal surfaces is still largely  open. In this paper
we address the issue of  regularity away from characteristic points. Our
goal is to prove the following
\begin{theorem}\label{maintheorem}
The Lipschitz continuous  vanishing viscosity solutions of \eqref{MAINgeq2} with
zero right-hand-side $f=0$
are smooth functions.
\end{theorem}

Invoking the implicit function theorem, we 
want to apply Theorem \ref{maintheorem} to the study of the
regularity   away from the characteristic locus of the Lipschitz
perimeter minimizers found in \cite{chy} for the case $\H^n$, $n>1$.
Here and in the following
$\nabla_E$ denotes the Euclidean gradient in $\R^{2n}$. We also denote by
$(y_1,...,y_{2n+1})$ exponential coordinates of the first kind\footnote{these are the coordinates used in \cite{chy}.}, defined by
$\exp(y_1 X_s+\sum_{i=1}^{2n-1} y_{i+1} X_i+y_{2n+1} \p_{2n})=\exp(s X_s)[\Pi_{i=1}^{2n-1} \exp(x_i X_i) ] \exp(x_{2n} \p_{2n}).$


\begin{corollary}
Let $O\subset \R^{2n}$ be a strictly  convex, smooth open set,
$\phi\in C^{2,\alpha}(\bar O)$ and
for each $(y_1,...,y_{2n})\in O$ denote by $(y_1,...,y_{2n})^*=(y_2,-y_1,y_4,-x_3,...)$. 
Consider the family  $$\{g_\e(y_1,...,y_{2n})\}_\e \ \  \ 
\sup_O |g_\e|+\sup_O|\nabla_E g_\e| \le C \ \ \  \text{(uniformly in }\e),$$
of   smooth  solutions of 
the approximating minimal surface PDE
$$div\Bigg(\frac{\nabla_E g_\e + (y_1,...,y_{2n})^*}{\sqrt{\e^2+|\nabla_E g_\e + (y_1,...,y_{2n})^*|}}\Bigg)=0 \text{ in }O \ \text{ and }g_\e=\phi \ \text{ in }\p O$$
found in \cite[Theorem 4.5]{chy}.  If for $p_0=(p_0^s,p_0^1,...,p_0^{2n-1})\in O$, $a>0$ and for every $\e>0$ we have $|\p_{y_1} g_\e(p_0)|>a>0$  (or any other partial derivative 
is non-vanishing at $p_0$ uniformly in $\e$) then 
there is a sequence $\e_k\to 0$ such that the Lipschitz perimeter minimizer $g=\lim_{\e_k\to 0}g_{\e_k}$ is smooth in a neighborhood
of the point $p_0$.
\end{corollary}

\begin{proof}
The implicit function theorem and a change of coordinates  imply that the level set of $$y_{2n+1}-g_\e(y_1,...,y_{2n})$$ can be written as  smooth intrinsic graphs $s=u_\e(x)$ in a neighborhood  of
$p_0$, with $u_\e$ defined in an open set $\Omega\subset \R^{2n}$. The Lipschitz bounds on $g_\e$ (proved in \cite[Propositions 4.2-4]{chy})
yield uniform Lipschitz bounds on $u_\e$, thus allowing to apply
Theorem \ref{maintheorem} and conclude the proof.
\end{proof}

We remark that in the case $n=1$  of the first Heisenberg group the 
regularity of vanishing viscosity minimal intrinsic graphs is quite different.
In the forthcoming paper \cite{CCM-1} we study  this problem and 
prove a form of intrinsic regularity, with differentiability along the Legendrian
foliation of the minimal graph.

\bigskip
Equation \eqref{MAINgeq2} is an uniformly elliptic approximation
of a subelliptic equations. The defining vector fields have
Lipschitz coefficients and satisfy a weak H\"ormander condition,
since together with their first order vector fields they span the
space at every point. The main difficulty of the proof is to
handle the vector field
$$ X_{2n-1, u}=\p_{2n-1}+ u(x)\p_{2n}$$
and the dependence on $\e$. A similar difficulty  arises in
problems of mathematical finance. For example in \cite{CPP},
\cite{CPP2}, it was proved that the viscosity solutions of the
following  equation are $C^\infty$
$$X_1^2 u + X_2 u=0 $$
where $X_1 = \p_{xx}$,$X_2 = \p_y u + u \p_z u$, satisfy a
weak H\"ormander condition analogue to the one in the present paper. The  techniques
in  \cite{CPP},
\cite{CPP2}, provide the main inspiration for the proof of Theorem \ref{maintheorem}.

\bigskip
The regularity of solutions  will be measured in terms of the natural 
norm of the   intrinsic
H\"older class $C^{1, \alpha}_u$, i.e.   functions $f$ such that
$\nabla^\e_u f$ is H\"older continuous, with respect to the
control distance $d_u$. The proof will be accomplished in two
steps:

\smallskip

{\bf STEP 1} First  prove that the Lipschitz continuous solutions
  are of class $C^{1, \alpha}_u$.
Since the operator $L_\e$ in (\ref{MAINgeq2e}) is represented in
divergence form, then by differentiating the PDE and combining several
horizontal and ``vertical" energy estimates,  it is possible to prove a  Euclidean
Cacciopoli-type inequality for the intrinsic gradient
$\nabla^{\e}_{ u }u$ of the solution. The Moser iteration technique
will then lead to H\"older continuous estimates uniform in $\e$ for the
gradient. This step holds also for $n=1$.

{\bf STEP 2} We prove the smoothness of the solution. In
order to do so we first note that the operator $L_\e$ can also be
represented in a divergence form:
\begin{equation}\label{nondiv}L_\e u= \sum_{i,j=1}^{2n} a^\e_{i,j}(\nabla^\e_u u)
X^\e_{i, u}X^\e_{j, u}u,\end{equation} where $$
a^\e_{ij}:R^{2n}\rightarrow R\quad  a^\e_{ij}(p)= \delta_{ij} -
\frac{p_ip_j}{1 + |p|^2}.$$ 
For every fixed point $x_0$ we
will approximate the assigned operator with a linear, uniformly
subelliptic operator in divergence form $L_{\e, x_0}$, with
$C^{\infty}$ coefficients. 
 The approximation is carried out through
a ad-hoc {\it freezing} technique, where the function $u$ in the coefficients of the vector field
is substituted with polynomials, in a technique reminiscent of the work of Rothschild and Stein \cite{RS}.  The novel difficulty arises from the non-smoothness
of $u$, and has to be dealt with through a delicate bootstrap argument.
The existence of a fundamental solution
$\Gamma^\e_{x_0}$ for such operator as well as its estimates, uniform in $\e$, have previously been proved in the papers
\cite{BLU} and \cite{CM}. Eventually  $\Gamma^\e_{x_0}$ will be used to
define a parametrix for the fundamental solution of $L_\e$ and to
obtain estimates independent of $\e$, of the derivatives of any
order of the solution.


\section{Notations and known results}

\subsection {H\"older classes}

In the sequel we will always keep fixed a function $\bar u \in
C^{\infty}=C^{\infty}(\Omega),$ with $\Omega\subset \R^{2n}$ and consider the vector fields $X^\e_{i, \bar u}$ in
(\ref{ecampi}), with coefficients depending on the fixed function
$\bar u$. Let us define a new vector field 
$$X^\e_{2n+1 , \u}=  \p_{2n}= [X^\e_{1, \u}, X^\e_{n, \u}], $$
 which act as a second order derivative, and  call degree of $\s_i$ the
natural number $deg(\s_i)=1$ for $\s_i\leq 2n$, $deg(2n+1)=2$.
Correspondingly the degree of any multi-index $\s =
(\s_{1},\dots,\s_{m})$, $\s_{r}\in \{1,\ldots, 2n+1\}$,  $1\le r
\le m\in \N$, will be:
$$deg(\s)=\sum_{ i=1}^mdeg(\s_i).$$
We will also denote the cardinality of
$\sigma=(\s_{1},\dots,\s_{m})$ the number of its elements:
$$\#(\s) = m.$$

 We define the intrinsic derivative
\begin{equation}\label{e40}
  \nabla^\e_{\s, \u} = X^\e_{\s_1, \u}\cdots X^\e_{\s_m, \u},
\end{equation}
and $\nabla^{\e  k}_{\bar u}$ the
 vector field with components
$(\nabla^{\e}_{\s   \bar u})_{deg(\s)=k}$.

\medskip
Since the vector fields $X^\e_{1, \u}, ...,X^\e_{2n, \u},$ are the
Riemannian completion of an H\"ormander type set of vectors, they
give rise to a control distance $d_{\e, \u}$. The corresponding   metric balls are denoted
$B_{\e, \u}(x,r)$. As $\e\to 0$ the metric space $(\Omega, d_{\e,\u})$
converge in the Gromov-Hausdorff sense to $(\Omega, d_\u)$ (see \cite{cdpt}).
%


%
\medskip

We next define the spaces of H\"older continuous functions related
to the fixed function $\u$.
\begin{definition}\label{d301}
Let $x_{0}\in\O$, $0<\a < 1$, assume that $\u$ is a fixed
Lipschitz continuous function, and that $u$ is defined on
$\Omega.$ We say that $u \in C_{\u}^{\a}(\Omega)$ if for every
compact set $K$ there exists a positive constant $M$ such that for
every $x, x_0\in K$ and $\e>0$
\begin{equation}\label{e301}
   |u(x) - u(x_{0})| \le M d_{\e, \u }(x, x_0).
\end{equation}
 Iterating this definition, if $k\geq 1$,  we say that $u \in
C_{\u}^{k,\a}(\O)$, if
 $\nabla^\e_{\u} u \in C_{\u}^{k-1,\a}(\O)$.
\end{definition}

\subsection{Taylor approximation.}

The following result is well know for vector fields with
$C^\infty$ coefficients (see \cite{NSW}) also holds for  vector
field is of the form $\p_{1} + \u \p_{2n}$, with $\u$ Lipschitz
continuous with respect to the Euclidean distance.  Let us first
denote by $e_1,\ldots , e_{2n} $ the canonical coordinates of a
point $x$ around $x_0$,
$$ x =exp(\sum_{i=1}^{2n-1} e_i X^\e_{i, \u} + e_{2n} X^\e_{2n+1
, \u})(x_0)$$
and, for a multi-index $\s=(\s_1,\ldots , \s_{2n})$ we will denote
$e_\s = (e_{\s_1},\ldots , e_{\s_{2n}})$. 

We explicitly note that,
since $X^\e_{2n, u}$ and $X^\e_{2n+1, u}$ are parallel, only one
of them can appear in the definition of canonical coordinates,
otherwise the values of $e_i$ would not be uniquely determined.
Due to this fact, for every multi-index $\s=(\s_1, \cdots \s_m),$
with components in $\{1, \cdots 2n \}$ we will denote $I(\s)=
(\r_1, \cdots \r_m)$, where $\r_i=\s_i$ if $\s_i\not= 2n$, and
$\r_i=2n+1$ if $\s_i=2n$.

\begin{theorem}\label{t301}
 Let $x_{0}\in\O$, $0<\a < 1, k \in \N \cup \{ 0
\}$ and assume that $u \in C_{\u}^{k,\a}(\O)$. Then we can define
Taylor polynomial of order $k$ the function
$$P^k_{x_0}u (x) = \sum_{h=1}^k \sum_{deg(\s) = h
, \atop \s_i \not=2n+1
}
\frac{1}{\#(\s)!} e_\s \nabla^\e_{ I(\s), \u}u(x_0)$$ and we have
\begin{equation}\label{e303}
  u(x) = P^k_{x_0}u (x) + O\left( d_{\e, \u}(x_0, x)^{k+\a}\right) \quad
  {\it as} \; \; x \to x_0.
\end{equation}
We will also set $P^{k}_{x_0}u=0$ for any  negative integer $k$.
\end{theorem}

Note that $$P_{x_0}^1u(x)=\sum_{i=1}^{2n-1} e_i(x)X_{i,\u}^\e u(x_0) +e_{2n}(x)X_{2n+1,\u}^\e u(x_0).$$

From the explicit expression of the Taylor polynomials of order
less than 4 it is possible to directly deduce the following
result.

\begin{remark} \label{r3062}
If $u\in C^{k,\a}_{\u}(\O)$, $0\le k\le 4$,   $K$ is a compact
subset of $\O$ and $\s$ is a multi-index, then there exists $C>0$
such that
  \begin{equation}\label{tosta}
  \left|P_{x_{0}}^{k}u(\x) - P_{x}^{k}u(\x)\right| \le C d_{\e, \u}^{\a}(x_{0},x)
  d_{\e, \u}^{k}(x_{0},\x),\end{equation}
for every $x, x_{0}, \x \in K$, see \cite[Lemma 3.6]{CMo} and \cite[Remarks  2.24 and 2.25]{CPP}.
\end{remark}

\subsection{Derivatives and Frozen derivatives.}
We will introduce here first order operators with polynomial
coefficients which locally approximate the vector fields $X^\e_{i,
\u}$. These new vector fields are defined in terms of the Taylor
development of the coefficients of $X^\e_{ i, \u}$. Precisely, for
any fixed point $x_0$ we will call operator frozen at the point
$x_0$
$$X^\e_{i, x_0}= X^\e_{i, \u} \text{ if  }
i \not= 2n-1, \quad X^\e_{2n-1, x_0} = \p_{2n-1} +
P^1_{x_0}\u(x)\p_{2n}$$ and for every multi-index $\sigma$,
$$  \nabla^\e_{\s, x_0}  = X^\e_{\s_1, x_0}\cdots X^\e_{
\s_m, x_0},$$ and $\nabla^{\e  k}_{x_0}$ will be the
 vector field with components
$(\nabla^{\e}_{ \s, x_0})_{deg(\s)=k}$.

\medskip
 These frozen derivatives have been defined as approximation of
the intrinsic derivatives, depending on $\u.$ In order to clarify
this point, we recall the following definition, given in \cite{FS}
and \cite{NSW}. If $\alpha \in R $ and $f(x, x_0)= O(d_{\e, \u
}^\alpha(x, x_0)) $ as $x\rightarrow x_0,$ we will say that the
differential operator $f(x, x_0)\nabla^{\e}_{\s, x_0}$ has degree
$deg(\s) -\alpha$. We have
$$X^\e_{i, \u} = X^\e_{i, x_0} \quad \text{ if } i\not= 2n-1$$
$$X^\e_{i, \u} = X^\e_{i, x_0} + (\u-P^1_{x_0} \u(x)) \p_{2n}\quad \text{
if } i = 2n-1$$ Hence, if $\u$ is of class $C^{1, \alpha}_\u$,
then $(\u-P^1_{x_0} \u(x))\p_{2n}$ is a differential operator of
degree $1-\alpha,$ while $X^\e_{i, \u}$ and $X^\e_{i, x_0}$ have
degree 1. This means that the intrinsic derivative is expressed as
the frozen derivative, plus a lower order term.

\medskip

More generally the following approximation result holds:

\begin{lemma} \label{duederiv}
If $\bar u\in C^{k-1, \alpha} _\u(\O)$,  and $\sigma$ is a
multi-index such that $deg(\s)\le k $, then for every function
$\phi\in C^\infty_0(\O)$ the derivative $\nabla^\e_{\s, \bar u}
\phi$ can be represented as
  \begin{equation}\begin{split} \nabla^\e_{\s, \bar u}\phi  = & \nabla^\e_{\s, x_0}\phi +
  \\&
  \!\!\!\!\!\!    \!\!\!\!\!\!    \sum_{deg(\rho)-h \leq deg(\s)}   \!\!\!\!\!\!   (\u - P^1_{x_0}\u)^h
    \!\!\!\!\!\!    \sum_{ deg(\mu_1)+\cdots +deg(\mu_k)\leq k-1\atop deg(\mu_k)\geq 0}
   \!\!\!\!\!\!  
      C_{\rho,\mu_i,\s,h}\prod_{ 1\leq deg(\mu_i)   }
        \!\!\!\!\!\!    \nabla^\e_{\mu, \u}(\u -
  P^1_{x_0}\u)\,  \nabla^\e_{\rho, x_0}\phi,
\end{split}\end{equation}
where  $C_{\rho,\mu,\s,h}$ are suitable constants. In particular the
operator $ \nabla^\e_{\s, \u}\phi(x) $ can be identified as a
differential operator of degree $deg(\s)$ and represented in terms
of frozen derivatives.
\end{lemma}
\medskip\noindent
\begin{proof}      Since the function $\phi$ is of class
$C^\infty_0(\O)$,  its Lie derivatives can be simply computed as
directional derivatives. By
definition\begin{equation}\label{duno}X^\e_{i,\bar u}\phi=X^\e_{i,
x_0}\phi +\delta_{i, 2n-1}(\u -
  P^1_{x_0}\u)\,  \p_{2n}\phi.\end{equation}
Hence the assertion is true if $deg(\s)=1$.
\medskip
If the assertion is true for any $\s$ such that $deg(\s)=k,$ then
we consider a multiindex $\s$ such that $deg(\s)=k+1$. In this
case
$$\s= (\s_1, \bar \s),$$ where
$$\left\{
\begin{array}{cc}
deg(\bar \s) = k & \text{ if }\s_1 \not= 2n+1\\
deg(\bar \s) = k-1 &  \text{ if } k\geq 3 \text{ and }\s_1 = 2n+1
\end{array}\right.$$
We have
 $$\nabla^\e_{\s, \u}\phi(x) = X^\e_{\s_1, \bar u}
\nabla^\e_{\bar \s, \u}\phi(x) =$$ by inductive assumption
$$=
 X^\e_{\s_1, \u}\Big(\nabla^\e_{\bar \s, x_0}\phi\Big)+
\!\!\!\!\!\   \!\!\!\!  \sum_{deg(\rho)-h \leq deg(\bar \s)}\!\!\!\!  X^\e_{\s_1 ,\u}\Big( (\u - P^1_{x_0}\u)^h
 \!\!\! \!\!\!\!\!\!  \sum_{ deg(\mu_1)+\cdots deg(\mu_k)\leq k-1\atop deg(\mu_k)\geq 0}
 \!\!\!\!\!\!\!\!\!\!   C_{\rho,\mu_i,\bar \s,h}
  \!\!\!\! 
  \prod_{ 1\leq deg(\mu_i)   }
   \!\!\!\! 
     \nabla^\e_{\mu, \bar u}(\u -
  P^1_{x_0}\u)\,  \nabla^\e_{\rho, x_0}\phi\Big),   $$
(also using (\ref{duno}))
 $$=
 X^\e_{\s_1, x_0} \nabla^\e_{\bar\s, x_0}\phi + \delta_{\s_1 2n-1}(\u - P^1_{x_0}\u)
 \p_{2n}\nabla^\e_{\bar\s, x_0}\phi$$$$+
\!\!\!\!\!\!  \sum_{deg(\rho)-h \leq deg(\bar\s)} (\u - P^1_{x_0}\u)^{h-1}
   X^\e_{\s_1, \u}(\u - P^1_{x_0}\u)\!\!\!\!\!
     \sum_{ deg(\mu_1)+\cdots deg(\mu_k)\leq k-1\atop
   deg(\mu_k)\geq 0}  \!\!\!\!\!
   C_{\rho,\mu_i,\bar\s,h}\prod_{ 1\leq deg(\mu_i)   }
   \nabla^\e_{\mu, \u}(\u -
  P^1_{x_0}\u)\,  \nabla^\e_{\rho, x_0}\phi+
$$$$+ \!\!\!
  \sum_{deg(\rho)-h \leq deg(\bar\s)}  \!\!\!\! (\u - P^1_{x_0}\u)^{h}
\!\!\!\!   \sum_{ deg(\mu_1)+\cdots +deg(\mu_k)\leq k-1\atop deg(\mu_k)\geq 0}
\!\!\!\!\!
 C_{\rho,\mu_i,\bar\s,h}
 X^\e_{\s_1, \u}\Big( \prod_{ 1\leq deg(\mu_i)   }  \nabla^\e_{\mu, \u}(\u -
  P^1_{x_0}\u)\Big)\,  \nabla^\e_{\rho, x_0}\phi+$$

  $$+
 \!\!\!\!  \sum_{deg(\rho)-h \leq deg(\bar\s)} \!\!\!\!  (\u - P^1_{x_0}\u)^{h}
  \sum_{ deg(\mu_1)+\cdots +deg(\mu_k)\leq k-1\atop deg(\mu_k)\geq 0}
   \!\!\!\!  C_{\rho,\mu_i,\bar\s,h}
  \prod_{ 1\leq deg(\mu_i)   }  \nabla^\e_{\mu, \u}(\u -
  P^1_{x_0}\u)\,  X^\e_{\s_1, x_0}\Big( \nabla^\e_{\rho, x_0}\phi\Big)=$$
 $$+
 \delta_{\s_1, 2n-1}  \!\!\!\!  \sum_{deg(\rho)-h \leq deg(\bar\s)}  \!\!\!\! 
 (\u - P^1_{x_0}\u)^{h+1}
  \!\!\!\! \sum_{ deg(\mu_1)+\cdots +deg(\mu_k)\leq k-1\atop deg(\mu_k)\geq 0}
   \!\!\!\!  C_{\rho,\mu_i,\bar\s,h}
  \prod_{ 1\leq deg(\mu_i)   }  \nabla^\e_{\mu, \u}(\u -
  P^1_{x_0}\u)\,  \p_{2n}\nabla^\e_{\rho, x_0}\phi.$$
Note that the first term satisfies
$$X^\e_{\s_1, x_0}
\nabla^\e_{\bar\s, x_0}\phi=\nabla^\e_{\s, x_0}\phi.$$ The second
term is $(\u - P^1_{x_0}\u) \p_{2n}\nabla^\e_{\bar\s, x_0}\phi$.
It can be considered one of the term listed in the thesis, with
$h=1$, while $\p_{2n}\nabla^\e_{\bar\s, x_0} = \nabla^\e_{\rho,
x_0},$ for a suitable $\r$, of degree $deg(\r)=k+2.$ Hence
$deg(\r)-h=k+1.$ Similarly, all the other terms are in the form,
indicated in the thesis (in both case, $\s_1 \not= 2n+1$ or $
k\geq 3$ and $\s_1 = 2n+1$).
\end{proof}
\bigskip

The vector fields $X^\e_{i, x_0}$ satisfy an H\"ormander type
condition, hence they define a control distance $d_{\e,
x_{0}}(x_{0},\x)$. The corresponding metric balls $B_{\e, x_0}(x,r)$ have volume
comparable to $r^{2n+1}$, and we will call
\begin{equation}\label{homodimen}Q=2n+1
\end{equation} the {\it  homogeneous dimension} of the space
$(\R^{2n},d_{\e, x_0})$. Note that the homogeneous dimension is
the same as the Hausdorff dimension of $(\R^{2n}, d_\u)$, defined in the introduction.
A simple modification of Proposition 2.4 in \cite{CPP} yields the following relation
between $d_{\e,x_0}$ and  the control distance $d_{\e, \u}$ associated to the vector
fields $X^\e_{i, \u}:$

\begin{proposition}\label{p201}
 For every compact subset $K$ of \ $\O$, there exists a positive
 constant $C = C(K)$ such that for every $x, x_{0},\ \in K$

    $C^{-1} d_{\e, x_{0}}(x_{0},\x) \le d_{\e, \u}(x_{0},\x) \le
 C d_{\e, x_0}(x_0,\x),$

   $d_{\e, x_{0}}(x_{0},x) \le C (d_{\e, x_{0}}(x_{0},\x) +
 d_{\e, \x}(\x, x)),$

  $d_{\e, \u}(x_{0},x) \le C (d_{\e, \u}(x_{0},\x) + d_{\e, \u}(\x, x)).$
\end{proposition}

\subsection{Linearized and frozen operator}

In analogy with the definition of linear vector fields, in terms
of a fixed function $\u$, we can also define a linearization
$L_{\e, \u}$ of the operator $L_\e$, written in terms of the
linearized vector fields $X^\e_{i, \u}$:

\begin{equation}\label{nondivlin}L_{\e, \u}u= \sum_{i,j=1}^{2n} a^\e_{i,j}(\nabla^\e_{\u}\u) X^\e_{i, \u}X^\e_{j, \u}u,\end{equation} where $a^\e_{i,j}$
are defined in (\ref{nondiv}). Since the function $\u$ is fixed,
the operator is a linear non divergence type operator, whose
coefficients have the regularity of the function $\u$. In case
$\u$ is not smooth, it is natural to approximate it with a frozen
operator, defined in term of the vector fields $X^\e_{i, x_0}$:
\begin{equation}\label{nondivfroz}L_{\e, x_0}u = \sum_{i,j=1}^{2n}
a^\e_{i,j}(\nabla^\e_{\u}\u(x_0)) X^\e_{i, x_0}X^\e_{j,
x_0}u,\end{equation} where $a^{\e}_{ij}$ are defined in
(\ref{nondiv}). This is a divergence form uniformly subelliptic
operator with $C^{\infty}$ coefficients, which depends on $\e$.
Hence it has a fundamental solution $\G^\e_{x_{0}}$ (see
\cite{BLU}), and its dependence on $\e$ which can be handled as in
\cite{CM}. Since $\G^\e_{x_{0}}$ depends on many variables, the
notation
 $$X^\e_{i, x_0}(x)\G^\e_{x_{0}}( \; \cdot \; ,\x)$$
shall denote the $X^\e_{i, x_0}$-derivative of $\G^\e_{x_{0}}(s,
\x)$ with respect to the variable $s$, evaluated at the point $x$.

\begin{theorem}\label{fundam}(\cite{CM} - Theorem 1.1) Let $x_{0}\in\O$.
 For every compact set $K\subset \Omega$  and for every $p\in \N$ there
  exist two positive
 constants $C, C_p$ independent of $\e$, such that
\begin {equation}\label{2}
|\nabla^\e_{\s,  x_0}(x)\G^\e_{x_0} (\cdot,\x)|\le C_{p }
\frac{d_{\e, x_0} ^{2-p}(x,\x)}{|B_{\e,x_0} (x, d_{\e, x_0}
(x,\x))|},\quad deg(\s) = p
\end {equation}
for every $x, \x\in K$ with $x\not= \x$,  where  $B_{\e,x_0}(x,r)$
denotes the ball with center $x$ and radius $r$ of the distance
$d_{\e,x_0}$. If $p =0$ we mean that no derivative are applied on
$\G^\e_{x_0}$.
\end{theorem}

\begin{remark}\label{remp410}
With the same notation as in preceding theorem, from Lemma
\ref{duederiv}, and inequality (\ref{2}) it follows that
\begin {equation}\label{2u}
|\nabla^\e_{\s, \u}(x)\G^\e_{x_0} (\cdot,\x)|\le C_{p }
\frac{d_{\e, x_0} ^{2-p}(x,\x)}{|B_{\e,x_0} (x, d_{\e, x_0}
(x,\x))|},\quad deg(\s) = p
\end {equation}
for every $x, \x\in K$ with $x\not= \x$.
\end{remark}

Hence, using Proposition 2.20 and 2.21 in (\cite{CPP}) we have:

\begin{proposition}\label{diffgamma}
Let $k\in \N$, $2\le k \le 6$. Let $\u \in C^{k-1,\a}_{\u}(\O)$
and $K$ be  a compact subset of $\O$. There is a positive constant $C$
independent of $\e$, such that
 \begin{align}\nonumber
  &\left|\left(\nabla^\e_{\s, \u}(x) -
  \nabla^\e_{\s,  \u}(x_{0})\right)\G^\e_{x_{0}}
 (\cdot,\x)\right|\\ \label{e452}
 & \quad \le C \left(d_{\e, x_{0}}(x_{0},x)d_{\e, x_{0}}(x_{0},\x)^{-Q-deg(\s)+1} +
 d_{\e, x_{0}}(x_{0},x)^{\a}d_{\e, x_{0}}(x_{0},\x)^{-Q-deg(\s)+2}\right),
 \end{align}
 and
 \begin{align}\nonumber
  &\left|\nabla^\e_{\s, \u}(x)\G^\e_{x}(\cdot,\x) -
  \nabla^\e_{\s, \u}(x)(x_{0})\G^\e_{x_{0}}(\cdot,\x)\right|\\ \label{e446}
  & \quad\le C \left(d_{\e, x_{0}}(x_{0},x)d_{\e, x_{0}}(x_{0},\x)^{-Q-deg(\s)+1} +
  d_{\e, x_{0}}(x_{0},x)^{\a}d_{\e, x_{0}}(x_{0},\x)^{-Q-deg(\s)+2}\right),
 \end{align}
 for every multi-index $\s$, $deg(\s)=k$, and for every
 $x,x_{0}\in K$ and $\x$ such that  $d_{\e, x_{0}}(x_{0},\x)\ge M
  d_{\e, x_{0}}(x_{0},x)$, for suitable $M>0$. The constant $Q$ is
  the homogeneous dimension of the space, defined in
  (\ref{homodimen}).
\end{proposition}

Estimates of this type for the fundamental solution are the key
elements used in Proposition 3.9 in \cite{CPP} to prove the
following result:


\begin{proposition}\label{kerestim}
Let $k\in \N$, $2\le k \le 4$. Assume that $u$ is a function of class $C^{k-1}_\u(\Omega)$
and that there are open sets $\Omega_1\subset
\subset\Omega_2\subset \subset\Omega_3\subset \subset \Omega$ such
that for every $x\in \Omega_1$ the function $u$ admits the
following representation
\begin{equation}\label{representation7}
 \begin{split}
 u(x)  = & \int\limits_{\O}  \G^\e_{x_{0}}(x,\x)
     N_{1} (\x,x_0 )  d\x + \int\limits_{\O}  \G^\e_{x_{0}}(x,\x)
  N_{2,k}(\x,x_0)  d\x
    \\+& \sum_{i=1}^{2n}
 \int\limits_{\O}X^\e_{i, x_0} \G^\e_{x_{0}}(x,\x)
N_{2,k i}(\x,x_0)  d\x + \int\limits_{\O}  \G^\e_{x_{0}}(x,\x)
N_{3,k}(\x,x_0)  d\x.
\end{split}
\end{equation}
 Also assume that
for every $x_0$ fixed $\in \Omega_1$, the kernels $N_i(\cdot,
x_0)$ are supported in $\overline{ \Omega_3}$, as functions of
their first variable and there exists a constant $C_1$ such that
the kernels satisfy the following conditions:

 (i) if $x_0$ is fixed $\in \Omega_1$ the
 $N_{1}(\cdot,x_{0})$  is supported in $\overline{\Omega_3-\Omega_2} $
 \begin {equation}\label{N1k}
  \left|N_{1}(\x,x_{0}) - N_{1}(\x,x)\right| \le C_1\, d_{\e, x_{0}}^{\a}(x_{0},x);
\end{equation}

\medskip
 (ii) $N_{2,k}(\cdot,x_0)$ and $N_{2,ki}(\cdot,x_0)$ are  smooth  functions  and   all derivatives  are
uniformly  H\"older continuous in the variable $x_0$, satisfying
condition (\ref{N1k}) with the same constant $C_1$ as $N_1$;

\medskip

 (iii) for every $\x\in \Omega_3$ and $x, x_0\in \Omega_1$
\begin {equation}\label{N3}|N_{3,k} (\x, x_0)|\leq C_1 d^{k-2+\alpha}_{\e, x_0}(x_0, \x),
\end{equation}
and
\begin {equation}\label{N3bis}|N_{3,k} (\x, x_0) - N_{3,k} (\x, x)|\leq
C_1d^\alpha_{\e, x_0}(x_0, x)d^{k-2}_{\e, x_0}(x_0,
\x).\end{equation} Then $u\in C^{k}_\u$ and for every $\s$ such
that $deg(\s)=k$

\begin{equation}\label{representation8}
 \begin{split}
\nabla^\e_{\s, \bar u} (u\phi)(x_0)  = & \int\limits_{\O}
\nabla^\e_{\s,  \u} (x_0) \G^\e_{x_{0}}(\cdot,\x)
     N_{1} (\x,x_0)  d\x
\\& +\int\limits_{\O}  \G^\e_{x_{0}}(\x,0)
  \nabla^\e_{\s, \u}  (x_0) \big( N_{2,k}(x_0\circ \x^{-1},x_0)
  \big)d\xi
\\&+ \sum_{i=1}^{2n}
  \int\limits_{\O} X^\e_{i, x_0}(x_0) \G^\e_{x_{0}}(\x,0)
  \nabla^\e_{\s,  \u}  (x_0) \big( N_{2,ki}(x_0\circ \x^{-1},x_0)
  \big)d\x
\\&+ \int\limits_{\O} \nabla^\e_{\s,  \u} (x_0)
  \G^\e_{x_{0}}(x,\x) N_{3,k}(\x,x_0)d\x.
\end{split}
\end{equation}
 Besides, for any    $ \a'<\a$,
  there exists a
constant $C$ only dependent on $C_1$ and on   $C_p$ in
(\ref{2})  such that
 $$||u||_{C^{k,\a ' }_\u}\leq C.$$
 \end{proposition}

\begin {remark} The derivatives  $\nabla^\e_{\s,  \u} (x_0)$ in (\ref{representation8})
can be computed by
 Lemma \ref{duederiv} in term of the frozen derivatives $\nabla^\e_{\s', x_0}(x_0)$. In particular  the frozen derivatives
 $\nabla^\e_{\s, x_0}(x_0) ( h(\cdot\circ \xi^{-1} ))$ can be calculated by  formula in
 Proposition 2.23 in  \cite{CPP}.

\end {remark}
\bigskip

\begin{remark} \label{osservazione}
It is not difficult to prove that the same result is still true,
if $u$ has a more general representation
\begin{equation}
 \begin{split}
 u(x)  = & \int\limits_{\O}  \G^\e_{x_{0}}(x,\x)
     N_{1} (\x,x_0 )  d\x + \int\limits_{\O}  \G^\e_{x_{0}}(x,\x)
  N_{2,k}(\x,x_0)  d\x
    \\+& \sum_{i=1}^{2n}
 \int\limits_{\O}X^\e_{i, x_0} \G^\e_{x_{0}}(x,\x)
N_{2,k i}(\x,x_0)  d\x + \int\limits_{\O}  \G^\e_{x_{0}}(x,\x)
N_{3,k}(\x,x_0)  d\x
  \\+& \sum_{i=1}^{2n}
 \int\limits_{\O}X^\e_{i, x_0} \G^\e_{x_{0}}(x,\x) N_{4,k i}(\x,x_0)  d\x,
\end{split}
\end{equation}
where the kernels $N_{4,k i}(\x,x_0)$ satisfy assumptions similar
to $N_{3,k }:$

\smallskip
for every $\x\in \Omega_3$ and $x, x_0\in \Omega_1$
\begin {equation} |N_{4,k i} (\x, x_0)|\leq C_1 d^{k-1+\alpha}_{\e x_0}(x_0, \x),
\end{equation}
and
\begin {equation} |N_{4,ki} (\x, x_0) - N_{4,ki} (\x, x)|\leq
C_1d^\alpha_{\e, x_0}(x_0, x)d^{k-1}_{\e, x_0}(x_0,
\x).\end{equation}
\end {remark}

\section{From $Lip$ to  $C^{1,\a}_{\u}$.}

Let us now start the first step in  the proof of the regularity result. Using in full
strength the nonlinearity of the operator $L_\e$, we prove here
some Cacciopoli-type inequalities for the intrinsic gradient of
$u$, and for the derivative $\p_{2n}u.$ The main novelty of the
proof is  that putting together two intrinsic
subelliptic Cacciopoli inequalities we will end up with an
Euclidean Cacciopoli inequality. In this way we can obtain the
H\"older-regularity of the gradient via a standard Moser
procedure.

\bigskip
We first observe that   $$\p_{2n} X^\e_{i, u} u=-{(X^\e_{i,u})}^*
\p_{2n}u,$$ where ${(X^\e_{i, u})}^*$ is the $L^2$ adjoint of the
differential operator $X^\e_{i, u}$ and
\begin{equation}\label{adjoint}{(X^\e_{i, u})}^*=-X^\e_{i, u},\quad\text{ if}\quad
i=1,...,2n-2,2n, \text{ and } \quad {(X^\e_{2n-1, u})}^* = -X^\e_{2n-1,
u}-\p_{2n}u.\end{equation}

We now  prove  that if $u$ is a smooth solution of $L_{\e}u=0$ in $\Omega \subset \R^{2n}$
then its   derivatives $\p_{2n} u$ and $X^\e_{k, u} u$ are
solution of a similar mean curvature type  equation with different right
hand side:

\begin{lemma}\label{equ2n}
If $u$ is a smooth  solution of $L_{\e}u=0$ then $\omega=\p_{2n} u + 2 ||u||_{Lip}$
is a  solution of the equation
\begin{equation}\label{2nderiv}\sum_{i,j }
{(X^\e_{i, u})}^*  \Big( \frac{a_{ij}(\nabla_u^\e u)}{\sqrt{1
+|\nabla_{u}^\e u|^2}} {(X^\e_{j, u})}^* \omega \Big) =0,
\end{equation}
where $a_{ij}$ are defined in (\ref{nondiv}).
\end{lemma}

\begin{proof}
Differentiating the equation $L_{\e}u=0$ with respect to $\p_{2n}$
we obtain
$$\p_{2n}\Big(X^\e_{i, u}
\Big(\frac{X^\e_{i, u}u}{\sqrt{1 + |\nabla_{ u}^\e u
|^2}}\Big)\Big) = 0$$ Using the previous remark
$${(X^\e_{i, u})}^*\Big(\p_{2n} \Big(\frac{X^\e_{i, u}u}{\sqrt{1 + |\nabla_{
u}^\e u|^2}}\Big)\Big) = 0$$
Note that $$\p_{2n}\Big(\frac{X^\e_{i, u}u}{\sqrt{1 + |\nabla_{
u}^\e u |^2}}\Big)= \frac{\p_{2n}X^\e_{i, u}u}{\sqrt{1 + |\nabla_{
u}^\e u |^2}} - \frac{X^\e_{i, u}u\,X^\e_{j, u}u \,\p_{2n}X^\e_{j,
u}u}{(1 + |\nabla_{u}^\e u |^2)^{3/2}}  $$
$$=-\frac{({X^\e_{i, u})}^*\p_{2n}u} {\sqrt{1 + |\nabla_{ u}^\e u |^2}} +
\frac{X^\e_{i, u}u\,X^\e_{j, u}u \,{(X^\e_{j, u})}^*\p_{2n}u}{(1 +
|\nabla_{ u}^\e u |^2)^{3/2}}.$$ The result  follows immediately.
\end{proof}

\bigskip
Differentiating the equation $L_{\e}u=0$  with respect to
$X^\e_{k, u}$ we obtain

\begin{lemma} If $u$ is a smooth  solution of $L_{\e}u=0$ then $z=X^\e_{k, u} u+ 2||u||_{Lip}$
with $k\leq 2n-1$
is a solution of the equation
\begin{equation}\label{PDEXu}\sum_{i,j}
X^\e_{i,u}  \Big( \frac{a_{ij}(\nabla^\e_u u)}{\sqrt{1
+|\nabla^\e_{u} u|^2}} X^\e_{j, u}z \Big) =$$$$=-\sum_{i }[X^\e_{k, u},
X^\e_{i, u}] \Big(\frac{X^\e_{i, u}u}{\sqrt{1 + |\nabla_{ u}^\e u
|^2}}\Big) - \sum_{i,j} X^\e_{i, u} \Big(\frac{a_{ij}(\nabla^\e_u u)}{\sqrt{1
+|\nabla^\e_{u} u|^2}} [X^\e_{k, u}, X^\e_{j, u}] u\Big),
\end{equation}
where $a_{ij}$ are defined in (\ref{nondiv}).
\end{lemma}

\bigskip

\begin{proof}
Differentiating the equation $L_{\e}u=0$ with respect to $X^\e_{k,
u}$ we obtain
$$X^\e_{k, u}\Big(X^\e_{i, u}
\Big(\frac{X^\e_{i, u}u}{\sqrt{1 + |\nabla_{ u}^\e u
|^2}}\Big)\Big) = 0$$
$$[X^\e_{k, u}, X^\e_{i, u}]
\Big(\frac{X_{i, u}u}{\sqrt{1 + |\nabla_{  u}^\e u |^2}}\Big) +
X^\e_{i, u}\Big( X^\e_{k, u}\Big(\frac{X^\e_{i, u}u}{\sqrt{1 +
|\nabla_{ u}^\e u|^2}}\Big)\Big) = 0$$
Note that $$ X^\e_{k, u}\Big(\frac{X^\e_{i, u}u}{\sqrt{1 +
|\nabla_{ u}^\e u |^2}}\Big)= \frac{X^\e_{k, u}X^\e_{i,
u}u}{\sqrt{1 + |\nabla_{ u}^\e u |^2}} - \frac{X^\e_{i,
u}u\,X^\e_{j, u}u \,X^\e_{k, u}X^\e_{j, u}u}{(1 + |\nabla_{u}^\e u
|^2)^{3/2}}  =$$
 $$ = \frac{a_{ij}(\nabla^\e_uu)}{(1 +
|\nabla_{u}^\e u |^2)^{1/2}} ([X^\e_{k, u}X^\e_{j, u}]u + X^\e_{j,
u}z)
$$
concluding the proof.
\end{proof}

\bigskip


%





\begin{remark}\label{calcolobraket}
It is useful to compute explicitly the commutators that appear in the previous result.

If $k\le n-1$ and $i=k+n-1$  or $i\le n-1$ and $k=i+n-1$, then
$$[X^\e_{k, u}, X^\e_{i, u}]= {sign(k-i)}\p_{2n};$$

If $i=2n-1$ and $k<2n-1$, then 
$$ [X^\e_{k, u}, X^\e_{i, u}] = X^\e_{k, u}u \,
\p_{2n};$$

\medskip

If $k=2n-1$ and $i\neq 2n-1$  then
$$[X^\e_{k, u}, X^\e_{i, u}]= -X^\e_{i, u}u\p_{2n} $$

\medskip

If $k=2n$  and $i=2n-1$ then
$$ [X^\e_{k, u}, X^\e_{i, u}] = X^\e_{k, u}u \,
\p_{2n}.$$

All other commutators vanish.
As a consequence, if $u$ is a smooth solution of $L_\e u=0$ then
$|[X^\e_{k, u}, X^\e_{i, u}]|$ is always bounded by $1+||u||_{Lip}^2$.
\end{remark}

\bigskip

\begin{proposition} \label{cacciopolimisto}(First Cacciopoli type inequality  for $X^\e_{k, u}u$ 
) If $u$ is a smooth solution of $L_{\e}u=0$ in $\Omega \subset \R^{2n}$, and $z = X^\e_{k, u}u + 2
||u||_{Lip}$ with $k\leq 2n$ then for every $p\neq 2$ there exists a
constant $C$, only dependent on the bounds on the spatial gradient
and on $p$  such that for every $\phi \in C^{\infty}_0$

$$\int |\p_{2n}z |^2 {z}^{p-2} \phi^2 \leq
C\, \left(\int |\nabla_u^\e z |^2 z^{p-2} \phi^2 + \int z^p(\phi^2
+ |\nabla_u^\e\phi|^2)\right).$$

The constant $C$ is bounded if $p$ is bounded away from $2$.
 If $p=2$ the 
inequality holds in the form
$$
\int |\p_{2n}z |^2  \phi^2 \leq
C\,  \int z^p(\phi^2
+ |\nabla_u^\e\phi|^2).$$

\end{proposition}
\begin{proof}  Calling $ \omega=\p_{2n} u + 2
||u||_{Lip}$, we have
 $$\int |\nabla_u^\e
\omega |^2 z^{p-2} \phi^2 \leq  C\, \int \frac{a_{ij}(\nabla_u^\e
u)}{\sqrt{1 +|\nabla_{u}^\e u|^2}}X^\e_{i, u}\omega X^\e_{j, u}
\omega z^{p-2}\phi^2=$$
$$=-C\int \frac{a_{ij}(\nabla_u^\e u)}{\sqrt{1
+|\nabla_{u}^\e u|^2}}(X^\e_{i, u})^*\omega \,X^\e_{j, u} \omega\,
z^{p-2}\,\phi^2-C\int \frac{a_{2n-1\, j}(\nabla_u^\e u)}{\sqrt{1
+|\nabla_{u}^\e u|^2}} X^\e_{j, u} \omega\,  \omega \p_{2n}u
\,z^{p-2}\phi^2=$$ (integrating by parts $X^\e_{j, u}$ in the
first integral )
$$=  C\int (X^\e_{j, u})^*\Big(
\frac{a_{ij}(\nabla_u^\e u)}{\sqrt{1 +|\nabla_{u}^\e
u|^2}}(X^\e_{i, u})^*\omega \Big) \omega  z^{p-2}\phi^2 + C\int
\frac{a_{2n-1\, j}(\nabla_u^\e u)}{\sqrt{1 +|\nabla_{u}^\e
u|^2}}(X^\e_{i, u})^*\omega  \, \omega \p_{2n}u \, z^{p-2}\phi^2$$
$$+ C(p-2)\int
\frac{a_{ij}(\nabla_u^\e u)}{\sqrt{1 +|\nabla_{u}^\e
u|^2}}(X^\e_{i, u})^*\omega \ \omega X^\e_{j, u} z \, z^{p-3}\phi^2 +
2C\int \frac{a_{ij}(\nabla_u^\e u)}{\sqrt{1 +|\nabla_{u}^\e
u|^2}}(X^\e_{i, u})^*\omega \ \omega  z^{p-2}\phi X^\e_{j, u}\phi $$
$$-C\int
\frac{a_{2n-1\,j}(\nabla_u^\e u)}{\sqrt{1 +|\nabla_{u}^\e u|^2}}
X^\e_{j, u} \omega \,\omega\p_{2n}u z^{p-2}\phi^2.$$ The first
integral vanishes by Lemma \ref{equ2n}. In the other integrals we
can use the fact that $$\Big|\frac{a_{ij}(\nabla_u^\e u)}{\sqrt{1
+|\nabla_{u}^\e u|^2}}\Big|\leq 1 \quad \text{ and } |(X^\e_{i,
u})^*\omega|  \leq (1+||u||_{Lip})|\nabla^\e_{u}\omega|$$ where
$||u||_{Lip}$ is bounded uniformly in $\e$ by assumption. Then
 $$\int |\nabla_u^\e
\omega |^2 z^{p-2} \phi^2 \leq C\Big(\int |\nabla_u^\e \omega | z^{p-2}
(\phi^2+ |\phi \nabla^\e_{ u}\phi|) + (p-2)\int |\nabla_u^\e \omega |
|\nabla_u^\e z | z^{p-3} \phi^2\Big)$$ (by H\"older inequality and the
fact that $z$ is bounded away from 0)
$$
\leq\delta \int |\nabla_u^\e \omega |^2z^{p-2}\phi^2 + C(\delta)
\int |\nabla_u^\e z |^2z^{p-2}\phi^2+ C(\delta) \int z^p (\phi^2 +
|\nabla_u^\e \phi|^2).$$ For   $\delta$  sufficiently small this
implies that \begin{equation}\label{XP2n}\int |\nabla_u^\e \omega
|^2 z^{p-2} \phi^2 \leq C \int |\nabla_u^\e z |^2z^{p-2}\phi^2+ C
\int z^p (\phi^2 + |\nabla_u^\e \phi|^2)\end{equation}
The constant $C$ above is bounded as long as $p$ is away from 2.
The special case $p=2$ follows along similar computations.
 
Note that, if  $k\not= 2n-1,$ we have  $$|\p_{2n} z| =
  |\p_{2n} X^\e_{k, u}u|=|X^\e_{k, u}\p_{2n} u|
  \leq|\nabla_u^\e \omega|$$
If $k = 2n-1,$ we have $$|\p_{2n} z| = |\p_{2n} X^\e_{2n-1, u}u|=
  |X^\e_{2n-1, u}\p_{2n} u| + |\p_{2n} u|^2
  \leq|\nabla_u^\e \omega| +  |\p_{2n} u|^2$$
  Hence, using again the boundness of $|\p_{2n} u|$ and the fact
  that $z$ is bounded from below, together with inequality
  (\ref{XP2n}) we conclude the proof.
\end{proof}

\bigskip

\begin{proposition}\label{Xcacciopoli} (Intrinsic Cacciopoli type inequality for $X^\e_{k, u}u$ )
If $u$ is a smooth  solution of $L_{\e}u=0$ in $\Omega \subset \R^{2n}$ and $z= X^\e_{k, u}u + 2
||u||_{Lip} $, with $k\leq 2n$, then for every $p\neq 1$ there exists
a constant $C>0$, only dependent on the bounds on the spatial
gradient and on $p$ such that for every $\phi \in C^{\infty}_0$
$$\int |\nabla_u^\e z |^2 z^{p-2} \phi^2 \leq C \int z^p (\phi^2 +
|\nabla_u^\e \phi|^2+  |\phi   \p_{2n} \phi|).$$
The constant $C$ is bounded if $p$ is bounded away from the values $1$ and $2$.
\end{proposition}
\begin{proof} Multiplying  the equation (\ref{PDEXu})  by $z^{p-1}\phi^2$ and integrating we obtain
$$\int X^\e_{i, u}  \Big( \frac{a_{ij}(\nabla^\e_u u)}{\sqrt{1
+|\nabla^\e_{u} u|^2}} X^\e_{j, u}z \Big)z^{p-1}\phi^2 =$$$$=-\int
[X^\e_{k, u}, X^\e_{i, u}] \Big(\frac{X^\e_{i, u}u}{\sqrt{1 +
|\nabla_{ u}^\e u |^2}}\Big) z^{p-1}\phi^2 - \int  X^\e_{i, u}
\Big(\frac{a_{ij}(\nabla^\e_u u)}{\sqrt{1 +|\nabla^\e_{u} u|^2}}
[X^\e_{k, u}, X^\e_{j, u}] u\Big) z^{p-1}\phi^2.$$

We denote by  $I_1$ and $I_2$ the integrals in the  right hand side.
Let us consider the left hand side
$$\int {X^\e_{i, u}}\Big(
\frac{a_{ij}(\nabla^\e_u u)}{\sqrt{1 +|\nabla^\e_{u} u|^2}}
X^\e_{j, u}z \Big) z^{p-1}\phi^2 =$$ (since $({X^\e_{i, u}})^* = -
X^\e_{i, u} - \delta_{i, 2n-1} \p_{2n}u$)
$$=-(p-1)\int   \frac{a_{ij}(\nabla^\e_u u)}{\sqrt{1
+|\nabla^\e_{u} u|^2}} X^\e_{j, u}z \, X^\e_{i, u}z \, z^{p-2}\,
\phi^2 - 2 \int \, \frac{a_{ij}(\nabla^\e_u u)}{\sqrt{1
+|\nabla^\e_{u} u|^2}} X^\e_{j, u}z \, z^{p-1}\,\phi \, X^\e_{i,
u}\phi-$$
$$-\int   \frac{a_{2n-1 j}(\nabla^\e_u u)}{\sqrt{1
+|\nabla^\e_{u} u|^2}} X^\e_{j, u}z \, \p_{2n}u\,    \, z^{p-2}\,
\phi^2.$$
Using the uniform ellipticity of $a_{ij}$ and the boundeness of
$a_{ij}$ and $\p_{2n}u$,  we obtain
$$
\int
|\nabla_u^\e z|^2 z^{p-2}\phi^2 \leq C\Big(
 2 \int \, | \nabla^\e_{ u}z| \, z^{p-1}\,\phi \,
|\nabla^\e_{ u}\phi|+\int   | \nabla^\e_{ u}z|    \, z^{p-2}\, \phi^2\Big)+
|I_1| + |I_2|.$$ From here, using  an H\"older inequality and the
boundeness of $z$ from below one has
\begin{equation}\label{dis}
\int |\nabla_u^\e z|^2 z^{p-2}\phi^2 \leq  C \int      \, z^{p-2}\,
(\phi^2+ |\nabla_u^\e \phi|^2)+ |I_1| + |I_2|.\end{equation}
 Next we estimate separately the terms $I_1$ and $I_2$. We begin with the latter and observe that  integrating by parts the expression of $I_2$, we have $$I_2 = (p-1)\int
 \frac{a_{ij}(\nabla^\e_u u)}{\sqrt{1 +|\nabla^\e_{u} u|^2}}
[X^\e_{k, u}, X^\e_{j, u}] u \,X^\e_{i, u}z\, z^{p-2}\,\phi^2 $$
$$+
2\int \frac{a_{ij}(\nabla^\e_u u)}{\sqrt{1 +|\nabla^\e_{u} u|^2}}
[X^\e_{k, u}, X^\e_{j, u}] u \, z^{p-1}\, \phi\, X_{i, u}\phi $$
$$+\int  \frac{a_{2n-1\, j}(\nabla^\e_u u)}{\sqrt{1
+|\nabla^\e_{u} u|^2}} [X^\e_{k, u}, X^\e_{j, u}] u \,\p_{2n}u \,
z^{p-1}\, \phi^2$$ (using the fact that both $a_{ij}$ and the
brakets, computed in Remark \ref{calcolobraket},  are bounded)

$$\leq
C\int |\nabla_u^\e z| z^{p-2}\phi^2 + C\int
z^{p-1}\phi|\nabla_u^\e \phi| + C \int z^{p-1}\phi^2 \leq $$
(since $z$ is bounded from below)
$$ \leq
\delta \int |\nabla_u^\e z|^2 z^{p-2}\phi^2 + C(\delta)  \int z^p
(\phi^2 + |\nabla_u^\e \phi|^2)\label{disI2}, $$ for every $\delta
>0$.


In order to estimate $I_1$ we first consider separately the case
$k\not= 2n-1$. If $k\le n-1$ then   by Remark \ref{calcolobraket}
$$I_1 =-{sign(k-n)}(1-\delta_{k, 2n})\int \p_{2n}
\Big(\frac{X^\e_{k+n-1, u}u}{\sqrt{1 + |\nabla_{  u}^\e u
|^2}}\Big) z^{p-1}\,\phi^2-\int X^\e_{k, u}u\p_{2n}
\Big(\frac{X^\e_{2n-1, u}u}{\sqrt{1 + |\nabla_{ u}^\e u |^2}}\Big)
z^{p-1}\phi^2\leq$$ (integrating by parts, using the boundness of
$\nabla_u^\e u$, and the fact that $\p_{2n}X^\e_{k, u}u = \p_{2n}
z $ )
$$\leq C\int |\p_{2n} z| z^{p-2} \phi^2 + C\int   z^{p-1}
|\phi\p_{2n} \phi| + C\int |\p_{2n} z| z^{p-1} \phi^2 \leq  $$
(using H\"older inequality and the fact that $z$ is bounded away
from 0)
$$\leq \delta \int |\p_{2n} z|^2 z^{p-2} \phi^2 + C(\delta) \int
z^{p} (|\phi\p_{2n} \phi| + \phi^2) \leq$$
  (by Proposition  \ref{cacciopolimisto})
 $$\leq \delta\int |\nabla_u^\e z|^2 z^{p-2} \phi^2 + C(\delta)
\int   z^{p}( |\phi \p_{2n} \phi|+ \phi^2 + |\nabla_u^\e
\phi|^2).$$ This estimate can be proved with a similar argument in the case
$n\le k \le 2n-2$ and $k=2n$. Hence if $k\not= 2n-1$ the conclusion follows by choosing
$\delta$ sufficiently small.

\medskip
If $k=2n-1 $, then   Remark
\ref{calcolobraket} yields
$$I_1 =-\int X^\e_{i, u} u \p_{2n} \Big(\frac{X^\e_{i, u}u}{\sqrt{1
+ |\nabla_{  u}^\e u |^2}}\Big) z^{p-1}\phi^2 = $$ (directly
computing the derivative with respect to $\p_{2n}$ )
$$=-\int
(X^\e_{i, u}u)^2 \p_{2n} \Big(\frac{1}{\sqrt{1 + |\nabla_{  u}^\e
u |^2}}\Big) z^{p-1}\phi^2 -\int X^\e_{i, u}u
\frac{\p_{2n}X^\e_{i, u}u }{\sqrt{1 + |\nabla_{ u}^\e u |^2}}
z^{p-1}\phi^2    $$ (since $\sum_i(X^\e_{i, u}u)^2 = |\nabla_{
u}^\e u |^2 $ )
$$=-\int
|\nabla_u^\e u|^2 \p_{2n} \Big(\frac{1}{\sqrt{1 + |\nabla_{ u}^\e
u |^2}}\Big) z^{p-1}\phi^2 -\frac{1}{2}\int
\frac{\p_{2n}|\nabla_u^\e u|^2  }{\sqrt{1 + |\nabla_{ u}^\e u
|^2}} z^{p-1}\phi^2.$$
 If  we set $F(s) = \frac{1}{\sqrt{1 +s}}$, then one easily computes $$\p_{2n}F(|\nabla_u^\e u|^2) = \Big(|\nabla^\e_u
u|^2\p_{2n}\Big(\frac{1}{\sqrt{1 + |\nabla_{ u}^\e u |^2}}\Big)+
\frac{1}{2} \frac{\p_{2n}|\nabla_u^\e u|^2}{\sqrt{1 + |\nabla_{
u}^\e u |^2}}\Big)$$so that  the previous integral becomes:

$$I_1= -\int \p_{2n}F(|\nabla_u^\e u|^2)z^{p-1}\phi^2 =$$
(integrating by parts)
$$ =(p-1)\int  F(|\nabla_u^\e u|^2)\p_{2n}z z^{p-2}\phi^2
 + 2 \int  F(|\nabla_u^\e u|^2)  z^{p-1}\phi \p_{2n}\phi \leq$$
  (using the fact that $F$ is bounded)
$$\leq C\int | \p_{2n}z|z^{p-2} \phi^2 + C \int
z^{p-1} |\phi\p_{2n} \phi |  \leq$$ (by Proposition
\ref{cacciopolimisto} and an  H\"older inequality)
 $$\leq \delta\int |\nabla_u^\e z|^2 z^{p-2} \phi^2 + C(\delta)
\int   z^{p}( |\phi\p_{2n}\phi | + \phi^2 + |\nabla_u^\e
\phi|^2),$$ thus concluding the proof.
\end{proof}

Next, we note that, from   Propositions \ref{cacciopolimisto} and
\ref{Xcacciopoli}  one can derive a Euclidean Cacciopoli type inequality 
for $z= X^\e_{k, u}u+ 2||u||_{Lip}$,  with $k\leq 2n$. Here and in the following
$\nabla_E$ denotes the Euclidean gradient in $\R^{2n}$.

\begin{proposition}\label{EuCacciopoli} (Euclidean Cacciopoli inequality)
If $u$ is a Lipschitz continuous solution of $L_{\e}u=0$ in $\Omega \subset \R^{2n}$, and $z=
X^\e_{k, u}u + 2||u||_{Lip}$, with $k\leq 2n$, then for every
$p\neq 1$ there exists a constant $C$, only dependent on the bounds on
the spatial gradient and on $p$ such that for every $\phi \in
C^{\infty}_0$
\begin{equation}\label{Euc-Cacc}
\int |\nabla_E z |^2 z^{p-2} \phi^2 \leq C \int z^p (\phi^2 +
|\nabla_E \phi|^2).\end{equation}
The constant $C$ is bounded if $p$ is bounded away from the values $1$ and $2$.
\end{proposition}

\begin{proof} Observe that there exists $C>0$ depending only
on $||u||_{Lip}$ and $\Omega$ such that for all points in $\Omega$,
 $$|\nabla_E z|^2 \le  C \Big(\sum_{k<2n} |X_{k,u}^\e z|^2+|\p_{2n} z|^2\Big).$$ Hence using Propositions \ref{cacciopolimisto} and
\ref{Xcacciopoli}  and observing that $$|\nabla_u^\e\phi|^2+ |\p_{2n}\phi|^2 \le C |\nabla_E \phi|^2$$ we obtain \eqref{Euc-Cacc}.
\end{proof}

From Proposition \ref{EuCacciopoli}, using the classical Moser
procedure in the Euclidean setting, we can immediately deduce the
following regularity result:

\begin{proposition}\label{Calphaestimate}
Let  $u$ be a solution of $L_{\e}u=0$  in $\Omega \subset \R^{2n}$ and set   $z= X^\e_{k, u}u +2
||u||_{Lip}$, with $k\leq 2n$. For every compact set
$K\subset\subset \Omega$ then there exist a real number $\alpha$
and a constant $C$, only dependent on the bounds on the spatial
gradient and on the choice of the compact set such that
$$ ||z||_{C^{\alpha}_u(K)}\leq C.$$
In particular we have  the estimate
$$\sum_{i=1}^{2n+1} \sum_{j=1}^{2n}||X_{i,u}^\e X_{j,u}^\e u||_{L^2(K)}+  ||u||_{C^{1,\alpha}_u(K)}\leq C.$$
\end{proposition}
\begin{proof}
For $p\neq 1,2$ and $z$ as in the statement of the proposition define the function
$$w=\begin{cases}
z^{\frac{p}{2}} & \text{ if }p\neq 0;\\  \ln z  & \text{ if }p=0.
\end{cases}$$
If $p\neq 0$ then
the Caccioppoli inequality \eqref{Euc-Cacc} and the Euclidean Sobolev embedding Theorem  yield
\begin{equation}\label{cac1}
\Big(\int |\phi w|^{2\theta}\Big)^{\frac{1}{\theta}} \le C p^2 \int w^2 |\nabla_E \phi|^2,
\end{equation}
for some $\theta>1$.  Let $0<r_1<r_2$  be sufficiently small so that the 
Euclidean ball $B_{r_2}$ is contained in $\Omega$. With an appropriate choice
of test function \eqref{cac1} implies

\begin{eqnarray}\label{cac2}\Big( \int_{B_{r_1}} |z|^{\theta p} \Big)^{\frac{1}{p\theta} }
 \le \Big(\frac{C p^2}{r_2-r_1}\Big)^{\frac{2}{p}}
\Big( \int_{B_{r_2}} z^p\Big)^{\frac{1}{p}} && \text{ if }p>0\\
\Big( \int_{B_{r_1}} |z|^{\theta p} \Big)^{\frac{1}{p\theta} }
\Big(\frac{C p^2}{r_2-r_1}\Big)^{\frac{2}{|p|}}  \ge
\Big( \int_{B_{r_2}} z^p\Big)^{\frac{1}{p}}. && \text{ if }p<0.
\end{eqnarray}

If $p=0$,  \eqref{Euc-Cacc}  implies
$$\int|\phi\nabla_E w|^2 \le C\int|\nabla_E \phi|^2.$$
Let $r>0$ sufficiently small so that the Euclidean ball $B_r\subset \Omega$.
A standard choice of test function and H\"older inequality yield
$$\int_{B_r} |\nabla_E w|\le C R^{2n+1}.$$
Recalling that $\Omega\subset \R^{2n}$ and using Poincare' inequality
we obtain $w\in BMO(\Omega)$.

At this point, using \eqref{cac2}, the John-Nirenberg Lemma and following the 
standard Moser iteration process (see for instance \cite[Chapter 8]{GT})
we obtain the H\"older regularity of $z$. \end{proof}

\section{From $C^{1,\a}_{\u}$ to $C^{\infty}_{\u}$.}

In this section we will conclude the proof of the regularity
result. The section is organized in 3 steps. We fix a function
$\u$, and study solutions of the linearized equation
$$L_{\e, \u}u= \sum_{i,j=1}^{2n} a^\e_{i,j}(\nabla^\e_{\u}\u) X^\e_{i,
\u}X^\e_{j, \u}u,=0$$
 defined in (\ref{nondivlin}), and represented in
non-divergence form. The solutions $u$ well be represented in
terms of the fundamental solution $\G^\e_{x_0}$ of the
approximating operator $L_{\e, x_0}$, defined in (\ref{nondivfroz}):
$$L_{\e, x_0}u = \sum_{i,j=1}^{2n} a^\e_{i,j}(\nabla^\e_{\u}\u(x_0))
X^\e_{i, x_0}X^\e_{j, x_0}u,$$ where $a^{\e}_{ij}$ are defined in
(\ref{nondiv}). Since estimates of $\G^\e_{x_0}$ uniform in $\e$
are well known (and have been recalled in section 2), from these
representation formulas we will deduce a priori estimates for the
solution $u$, in terms of the fixed solution $\bar u$. Choosing
$\u=u$ we will obtain a priori estimates of the solutions of the
non linear equation $L_\e u=0$. Finally, letting $\e$ go to 0, we
will conclude the proof of the estimates of the vanishing
viscosity solutions of $Lu=0.$

\subsection{Representation formulas}

\begin{lemma}
The difference between the operator $L_{\e,
\bar u}$ and its frozen operator  can be
expressed as follows:

 \begin {equation}\label{LmenoL}
 \begin {split}
(L_{\e, x_0} - & L_{\e, \u})u(\x) =
\sum_{ij=1}^{2n}\Big(a_{ij}^\e(\nabla_\u^\e \u(x_0))-
a_{ij}(\nabla_\u^\e \u(\x))\Big)X^\e_{i,\bar u } X^\e_{j,\bar u
}u(\x)
\\&
-\sum_{j=1}^{2n}a_{2n-1 j}^\e(\nabla_\u^\e \u(x_0))(\bar u(\x)-
P^1_{x_0}\bar u(\x))\p_{2n}X^\e_{j, \bar u}u(\x)
\\&
- \sum_{j=1}^{2n}a_{2n-1 j}^\e(\nabla_\u^\e
\u(x_0))X^\e_{j,x_0}\Big((\bar u(\x)- P^1_{x_0}\bar u
(\x))\p_{2n}\Big)u(\x).
\end{split}\end{equation}
\end{lemma}
\begin{proof} Observe that 
$$
  (L_{\e, x_0} -  L_{\e, \u})u(\x) =
\sum_{ij=1}^{2n}\Big(a_{ij}^\e(\nabla_\u^\e \u(x_0))-
a_{ij}^\e(\nabla_\u^\e \u(\x))\Big)X^\e_{i,\bar u } X^\e_{j,\bar u
}u(\x) -
$$$$
   -\sum_{ij=1}^{2n}a_{ij}^\e(\nabla_\u^\e \u(x_0))\Big(X^\e_{i,\bar u }
X^\e_{j,\bar u }-X^\e_{i,x_0 } X^\e_{j,x_0 }\Big) u(\x) = $$$$ =
\sum_{ij=1}^{2n}\Big(a_{ij}^\e(\nabla_\u^\e \u(x_0))-
a_{ij}^\e(\nabla_\u^\e \u(\x))\Big)X^\e_{i,\bar u } X^\e_{j,\bar u
}u(\x) - $$$$-\sum_{ij=1}^{2n}a_{ij}^\e(\nabla_\u^\e
\u(x_0))\Big((X^\e_{i,\bar u }-X^\e_{i,x_0
   })X^\e_{j,\bar u}+X^\e_{i,x_0 }(
X^\e_{j,\bar u }-X^\e_{j,x_0 }) \Big) u(\x) $$$$=
\sum_{ij=1}^{2n}\Big(a_{ij}^\e(\nabla_\u^\e \u(x_0))-
a_{ij}^\e(\nabla_\u^\e \u(\x))\Big)X^\e_{i,\bar u } X^\e_{j,\bar u
}u(\x) - $$$$-\sum_{ij=1}^{2n}a_{ij}^\e(\nabla_\u^\e
\u(x_0))\Big(\delta_{i,2n-1}(\bar u(\x)- P^1_{x_0}\bar
u(\x))\p_{2n}X^\e_{j,\bar u} + X^\e_{i,x_0}(\delta_{j,2n-1}(\bar
u(\x)- P^1_{x_0}\bar u(\x))\p_{2n})\Big)u(\x)
$$
where $\delta $ is the Kroeneker function. From this the thesis
immediately follows
\end{proof}


Let us first represent the solutions of the equation $L_{\e,
\u}u=0$ it in terms of the fundamental solution $ \G^\e_{x_0}$ of
the operator $L_{\e, x_0}$ defined in (\ref{nondivfroz}).
\begin{proposition}\label{p402}
Let us assume that  $\bar u$ is a fixed function of class
$C^\infty(\O),$ and that $u$ is a classical solution of $L_{\e,
\bar u}u = g \in C^\infty(\O),$
  Then for any $\phi\in C^\infty_0(\O)$ the function 
     $u\phi$ can be represented as
 \begin{equation}\label{representation1}
\begin{split}
 u\phi(x) & =      \int\limits_{\O}  \G_{x_{0}}^\e(x,\x)
     N_{1} (\x,x_0)  d\x
+\int\limits_{\O}  \G_{x_{0}}^\e(x,\x)
   L_{\e, \bar u}u (\x)\,\phi(\x)\,d\x+
\\&+  \sum_{ij=1}^{2n}  \int\limits_{\O}  \G_{x_{0}}^\e(x,\x)\Big(b_{ij}^\e(\nabla_\u^\e \u(x_0))-
b_{ij}^\e(\nabla_\u^\e \u(\x))\Big)X^\e_{i,\bar u } X^\e_{j,\bar u
}u(\x)\phi(\x)d\x
\\&+ \sum_{ijs=1}^{2n}   \int\limits_{\O}
 X^\e_{s,x_0}
\G_{x_{0}}^\e(x,\x) (\bar u(\x)- P^1_{x_0}\bar
u(\x))h_{sij}(x_0)X^\e_{i,\bar u } X^\e_{j,\bar u
}u(\x)\phi(\x)d\x.
\end{split}\end{equation}
The expressions of $N_1,$ $b_{ij}$ and $h_{sij}$ are the
following:
\begin{equation}\label{defN1}
\begin{split} N_{1} (\x,x_0)&= \, u(\x)L_{\e, x_0}\phi(\x)  +
 \\
&+ \sum_{ij=1}^{2n}a_{ij}^\e(\nabla_\u^\e \u(x_0))
   \Big(X^\e_{i,\u }u(\x) X^\e_{j,x_0 }\phi(\x) + X^\e_{j,\u}u(\x) X^\e_{i,x_0}\phi(\x)\Big)\\
&-\sum_{i =1}^{2n} a_{i\,2n-1}^\e(\nabla_\u^\e \u(x_0))
(\bar u(\x) - P^1_{x_0}\bar u (\x))\p_{2n}  u (\x)\, X^\e_{i,x_0}\phi(\x) \\
&+ \sum_{i =1}^{2n} a_{i\,2n-1}^\e(\nabla_\u^\e \u(x_0))  (\bar
u(\x)- P^1_{x_0}\bar u(\x)) X^\e_{n, \bar u}  X^\e_{i, \bar u
}u(\x) X^\e_{1 \bar u}\phi(\x)
\\
&-\sum_{i=1}^{2n} a_{i\,2n-1}^\e(\nabla_\u^\e \u(x_0))  (\bar
u(\x) - P^1_{x_0}\bar u (\x)) X^\e_{1, \bar u}  X^\e_{i, \bar u
}u(\x) X^\e_{n, \bar u}\phi (\x).
\end{split}\end{equation}

$$b_{ij}:\R^{2n}\rightarrow \R,$$
\begin{equation}\label{bij}
b_{i,j}(p)= - \delta_{ik} a_{kj}(p) + a_{2n-1 j}(\nabla_\u^\e
\u(x_0))p_1 \delta_{in} - p_n  \delta_{i1} a_{2n-1 j}(\nabla_\u^\e
\u(x_0))
\end{equation}
Finally $h_{sij}$ are real numbers, only dependent on $x_0$,
defined as
\begin{equation}\label{hkij}
h_{sij}(x_0)=-
a^\e_{2n-1,s}(\nabla^\e_\u\u(x_0))(\delta_{i1}\delta_{jn} -
\delta_{in}\delta_{j1})+
a^\e_{2n-1,j}(\nabla^\e_\u\u(x_0))(\delta_{s1}\delta_{in} -
\delta_{sn}\delta_{i1}).
\end{equation}
\end{proposition}

\medskip\noindent
\begin{proof} By definition of fundamental solution, we have
\begin {equation}\label{prima}
  \begin {split}
   u\phi(x)  & = \int\limits_{\O} \G_{x_{0}}^\e(x,\x)
   L_{\e, x_0}(u\phi)(\x) d\x\\
   & =   \int\limits_{\O} \G_{x_{0}}^\e(x,\x)
   \left( u \,L_{\e, x_0}\phi + \sum_{ij=1}^{2n}a_{ij}^\e(\nabla_\u^\e \u(x_0))
   \Big(X^\e_{i,x_0 }u X^\e_{j,x_0 }\phi + X^\e_{j,x_0 }u X^\e_{i,x_0}\phi\Big)\right) d\x\\
   &+  \int\limits_{\O} \G_{x_{0}}^\e(x,\x)
   L_{\e, \u} u (\x)\, \phi(\x)  d\x  + \int\limits_{\O} \G_{x_{0}}^\e(x,\x)
   (L_{\e, x_0} - L_{\e, \u})u (\x)\, \phi(\x)  d\x.
 \end{split}\end{equation}

We can use the expression of $L_{\e, x_0} - L_{\e, \u}$ computed
in  (\ref{LmenoL}). Let us consider the second term in the right
hand side, multiplied by the fundamental solution. Since $\p_{2n}=
[X^\e_{1, x_0}, X^\e_{n, x_0}]$, it becomes:
\begin{equation}\label{quinta}
\sum_{j=1}^{2n}a_{2n-1 j}^\e(\nabla_\u^\e \u(x_0))\int\limits_{\O}
\G^\e_{x_{0}}(x,\x) (\bar u- P^1_{x_0}\bar u)\p_{2n} X^\e_{j, \bar
u }u \phi d\x\end{equation}
$$= \sum_{j=1}^{2n}a_{2n-1 j}^\e(\nabla_\u^\e \u(x_0))\int\limits_{\O}
 \G^\e_{x_{0}}(x,\x) (\bar u- P^1_{x_0}\bar u)
[X^\e_{1, x_0}, X^\e_{n, x_0}]\big( X^\e_{j, \bar u }u\big)\phi
d\x=$$ (integrating by part and using the fact that
 $X^\e_{1, x_0}= X^\e_{1 ,\bar u}$ and $X^\e_{n, x_0}= X^\e_{n,
\bar u}$.)
$$= - \sum_{j=1}^{2n}a_{2n-1 j}^\e(\nabla_\u^\e \u(x_0))\int\limits_{\O}X^\e_{1 x_0}
 \G^\e_{x_{0}}(x,\x) (\bar u(\x)- P^1_{x_0}\bar u(\x))
 X^\e_{n, \bar u}  X^\e_{j, \bar u
}u \phi\,d\x$$ $$-\sum_{j=1}^{2n}a_{2n-1 j}^\e(\nabla_\u^\e
\u(x_0))\int\limits_{\O} \G^\e_{x_{0}}(x,\x) X^\e_{1, \bar u}(\bar
u(\x)- P^1_{x_0}\bar u(\x)) X^\e_{n, \bar u} X^\e_{j, \bar u
}u\phi \,d\x$$
$$-\sum_{j=1}^{2n}a_{2n-1 j}^\e(\nabla_\u^\e \u(x_0))\int\limits_{\O}
  \G^\e_{x_{0}}(x,\x) (\bar u(\x)- P^1_{x_0}\bar u(\x))
X^\e_{n \bar u} X^\e_{j, \bar u }uX^\e_{1 \bar u}\phi \,d\x$$
$$+\sum_{j=1}^{2n}a_{2n-1 j}^\e(\nabla_\u^\e \u(x_0))
\int\limits_{\O}X^\e_{n, x_0} \G^\e_{x_{0}}(x,\x) (\bar u(\x)-
P^1_{x_0}\bar u(\x)) X^\e_{1 ,\bar u}  X^\e_{j, \bar u }u\phi
\,d\x$$
$$+\sum_{j=1}^{2n}a_{2n-1 j}^\e(\nabla_\u^\e \u(x_0))\int\limits_{\O}
\G^\e_{x_{0}}(x,\x)  X^\e_{n \bar u}(\bar u(\x)- P^1_{x_0}\bar
u(\x))  X^\e_{1, \bar u} X^\e_{j, \bar u }u \phi \,d\x$$
$$+\sum_{j=1}^{2n}a_{2n-1 j}^\e(\nabla_\u^\e \u(x_0))
\int\limits_{\O} \G^\e_{x_{0}}(x,\x)(\bar u(\x)- P^1_{x_0}\bar
u(\x)) X^\e_{1, \bar u}  X^\e_{j, \bar u }u X^\e_{n, \bar u}\phi
\,d\x.$$

The third   term  in (\ref{LmenoL}) becomes
\begin{equation}\label{formula3}
\begin{split}
 & \sum_{i=1}^{2n}  a_{2n-1 j}^\e(\nabla_\u^\e \u(x_0))\int\limits_{\O}
   \G^\e_{x_{0}}(x,\x) X^\e_{j, x_0}\Big((\bar u(\x)- P^1_{x_0}\bar u(\x))\p_{2n}\Big)u(\x)
\phi(\x)  d\x
\end{split}
\end{equation}
(integrating by part)
\begin{equation}
\begin{split}
  = &-  \sum_{j=1}^{2n} a_{2n-1 j}^\e(\nabla_\u^\e \u(x_0)) \int\limits_{\O}   X^\e_{j,x_0} \G^\e_{x_{0}}(x,\x) \,
(\bar u(\x)- P^1_{x_0}\bar u(\x))\p_{2n}  u(\x)\, \phi(\x) d\x \\
 &-  \sum_{j=1}^{2n}  a_{2n-1 j}^\e(\nabla_\u^\e \u(x_0))\int\limits_{\O}    \G^\e_{x_{0}}(x,\x) \,
(\bar u(\x)- P^1_{x_0}\bar u(\x))\p_{2n}  u(\x)\,
X^\e_{j,x_0}\phi(\x) d\x
\end{split}
\end{equation}

Inserting   (\ref{quinta}) and (\ref{formula3})   in
(\ref{LmenoL}) and using the expression of $b_{ij}$ and
$h_{sij}(x_0)$,
 we obtain

\begin{equation}\begin{split}
\int\limits_{\O} &\G_{x_{0}}^\e(x,\x)
   (L_{\e, x_0} - L_{\e, \u})u (\x)\, \phi(\x)  d\x
\\& =
 \sum_{ij=1}^{2n}  \int\limits_{\O}  \G_{x_{0}}^\e(x,\x)
\Big(a_{ij}^\e(\nabla_\u^\e \u(x_0))- a_{ij}^\e(\nabla_\u^\e
\u(\xi))\Big)X^\e_{i,\bar u } X^\e_{j,\bar u }u(\x)\phi(\x)d\x
\\&+  \sum_{ij=1}^{2n}  \int\limits_{\O}
\G_{x_{0}}^\e(x,\x)\Big(b_{ij}^\e(\nabla_\u^\e \u(x_0))-
b_{ij}^\e(\nabla_\u^\e \u(\x))\Big)X^\e_{i,\bar u } X^\e_{j,\bar u
}u(\x)\phi(\x)d\x
\\&+ \sum_{ijs=1}^{2n}   \int\limits_{\O}
 X^\e_{s,x_0}
\G_{x_{0}}^\e(x,\x) (\bar u(\x)- P^1_{x_0}\bar
u(\x))h_{sij}(x_0)X^\e_{i,\bar u } X^\e_{j,\bar u
}u(\x)\phi(\x)d\x
\end{split}\end{equation}
\begin{equation}\begin{split}
&
 -\sum_{j=1}^{2n}a_{2n-1 j}^\e(\nabla_\u^\e \u(x_0))\int\limits_{\O}
  \G^\e_{x_{0}}(x,\x) (\bar u(\x)- P^1_{x_0}\bar u(\x))
X^\e_{n \bar u} X^\e_{j, \bar u }uX^\e_{1 \bar u}\phi \,d\x
\\
 &
+\sum_{j=1}^{2n}a_{2n-1 j}^\e(\nabla_\u^\e \u(x_0))
\int\limits_{\O} \G^\e_{x_{0}}(x,\x)(\bar u(\x)- P^1_{x_0}\bar
u(\x)) X^\e_{1, \bar u}  X^\e_{j, \bar u }u X^\e_{n, \bar u}\phi
\,d\x
\\
 &-  \sum_{j=1}^{2n}  a_{2n-1 j}^\e(\nabla_\u^\e \u(x_0))\int\limits_{\O}    \G^\e_{x_{0}}(x,\x) \,
(\bar u(\x)- P^1_{x_0}\bar u(\x))\p_{2n}  u(\x)\,
X^\e_{j,x_0}\phi(\x) d\x
\end{split}\end{equation}

From this expression, equation (\ref{prima}), and the expression
of $N_1$ in (\ref{defN1}) we obtain the asserted representation
formula.
\end{proof}

\bigskip

The following representation formula will be used to estimate higher order derivatives
of the solutions.

\begin{proposition}\label{krappres}
Let us assume that  $\bar u$  is a fixed function of class
$C^{\infty}(\O)$, and assume that $u$ is a classical solution of
$L_{\e, \bar u}u = g \in C^{\infty}(\O)$. Then  for any $\phi\in C^{\infty}_0(\O)$ the
  function
$u\phi$ can be represented as
 \begin{equation}\label{representation2}
\begin{split}
 u\phi(x)  = & \int\limits_{\O}  \G^\e_{x_{0}}(x,\x)
     N_{1} (\x,x_0 )  d\x + \int\limits_{\O}  \G^\e_{x_{0}}(x,\x)
  N_{2,k}(\x,x_0) \phi(\xi) d\x
    \\+& \sum_{s=1}^{2n}
 \int\limits_{\O}X^\e_{s, x_0} \G^\e_{x_{0}}(x,\x)
N_{2,k s}(\x,x_0)  d\x + \int\limits_{\O}  \G^\e_{x_{0}}(x,\x)
N_{3,k}(\x,x_0)  d\x
  \\+& \sum_{i=1}^{2n}
 \int\limits_{\O}X^\e_{i, x_0} \G^\e_{x_{0}}(x,\x) N_{4,k i}(\x,x_0)  d\x,
\end{split}
\end{equation}
where $N_{1} (\x,x_0)$ is defined in (\ref{defN1}). If $b_{ij}$ is
the function defined in (\ref{bij}), we call $b_{ij\u}=
b_{ij}(\nabla^\e_{\u}\u),$ and the other kernel are expressed:
$$
N_{2} (\x,x_0) =  P^{k-2}_{x_0} g(\xi)  +\sum_{ij=1}^{2n}
  \Big(P^{k-2}_{x_0}b_{ij\u}(\x) -b_{ij\u}(x_0)\Big)P^{k-3}_{x_0}
   (X^\e_{i, \u }X^\e_{j, \u }u)(\x)$$
$$N_{2,k s}(\x,x_0)  =
 \sum_{ij=1}^{2n}\Big(P^{k-1}_{x_0}\bar u(\xi) -P^1_{x_0}\bar
   u(\xi)\Big)h_{sij }(x_0) P^{k-3}_{x_0}(X^\e_{i, \u }X^\e_{j, \u }u)(\x)
$$\begin{equation}\begin{split}
N_{3,k}(\x,x_0) & = \Big(g(\x)-P^{k-2}_{x_0} g(\x)\Big) +\\ +&
\Big(b_{ij\u}(\x) - b_{ij\u}(x_0)\Big)\Big(X^\e_{i, \u }X^\e_{j,
\u }u(\x) -P^{k-3}_{x_0} (X^\e_{i, \u }X^\e_{j, \u }u)(\x)\Big)+\\
+&
 \Big(b_{ij\u}(\x) - P^{k-2}_{x_0}b_{ij\u}(\x)\Big)P^{k-3}_{x_0}(X^\e_{i, \u }X^\e_{j, \u
 }u)(\x)\end{split}
\end{equation}
 \begin{equation}\begin{split}
N_{4,ks}(\x, x_0)=&
 \Big(\bar u(\x)- P^{k-1}_{x_0}\bar u(\x)\Big)h_{sij}(x_0)P^{k-3}_{x_0}
 (X^\e_{i, \u }X^\e_{j, \u
 }u)(\x)
+\\ +&
 \Big(\bar u (\x)- P^{1}_{x_0}\bar u(\x)\Big)h_{sij}(x_0)\Big(X^\e_{i, \u }X^\e_{j, \u
 }u(\x) -P^{k-3}_{x_0}(X^\e_{i, \u }X^\e_{j, \u
 }u)(\x)\Big).
\end{split}
\end{equation}
\end{proposition}

\begin{proof} We represent $u\phi$ as in  formula
(\ref{representation1}), and we study each term separately. Let us
start with the  second term in (\ref{representation1}):
\begin{equation}\label{kformulatermine1}
g(\x)=L_{\e, \bar u}u(\x)= P^{k-2}_{x_0} g(\x) +
(g(\x)-P^{k-2}_{x_0} g(\x))
 \end{equation}
The kernel in the third term of (\ref{representation1}) will
developed as follows:
\begin{equation}\label{kformulatermine2}
\begin{split}
 \Big(b_{ij\u}(\x)- & b_{ij\u}(x_0)\Big)X^\e_{i, \u }X^\e_{j, \u }u(\x)   =
\\ = & \Big(P^{k-2}_{x_0}b_{ij\u}(\x) -b_{ij\u}(x_0)\Big)P^{k-3}_{x_0}(X^\e_{i, \u }X^\e_{j, \u }u)(\x)
+ \\ + &   \Big(b_{ij\u}(\x) - b_{ij\u}(x_0)\Big)
 \Big(X^\e_{i, \u }X^\e_{j, \u }u(\x)-P^{k-3}_{x_0}(X^\e_{i, \u }X^\e_{j, \u }u)(\x)\Big)
   \\ + &  \Big(b_{ij\u}(\x) - P^{k-2}_{x_0}b_{ij\u}(\x)\Big)P^{k-3}_{x_0}(X^\e_{i, \u }X^\e_{j, \u
   }u)(\x).
   \end{split}
   \end{equation}
The first terms in (\ref{kformulatermine1}) and
(\ref{kformulatermine2}) define $N_{2, k}$, the sum of the other
terms defines $N_{3, k}$.

 The kernel in the last term of (\ref{representation1})
can be represented as
\begin{equation}\label{kformulatermine3}
\begin{split}
\Big(\bar u(\x)- & P^1_{x_0}\bar u(\x)\Big)h_{kij}(x_0)X^\e_{i, \u
}X^\e_{j, \u }u(\x)   = \\=&\Big(\bar u(\x) - P^1_{x_0}\bar
u(\x)\Big)h_{sij}(x_0)\Big(X^\e_{i, \u }X^\e_{j, \u
}u(\x)-P^{k-3}_{x_0}(X^\e_{i, \u }X^\e_{j, \u }u)(\x)\Big)
   \\ + &  \Big(\bar u(\x) - P^{k-1}_{x_0}\bar u(\x)\Big)h_{sij}(x_0)P^{k-3}_{x_0}
   (X^\e_{i, \u }X^\e_{j, \u }u)(\x) +\\+&
   \Big(P^{k-1}_{x_0}\bar u(\x) -P^1_{x_0}\bar u(\x)\Big)h_{sij}(x_0)P^{k-3}_{x_0}
   (X^\e_{i, \u }X^\e_{j, \u }u)(\x).\\
   \end{split}
   \end{equation}
The first two terms of this expression define $N_{4, ks}$, the
third defines $N_{2, ks}$.
\end{proof}

\subsection{A priori estimates of the solution of the linear operator }


\begin{proposition} \label{esistderivL2}
Assume that $u$ and  $\bar u$ are of class
$C^{\infty}_{\u}(\O)$ and that $L_{\e, \bar u}
u=g\in C^{\infty}_{\u}(\O)$. Also assume that there exists a
compact set $K\subset \O$ and constant $C_0$ such that
$$||\u||_{C^{1, \alpha}_\u(K)} + ||g||_{C^{1, \alpha}_\u(K)} +
\sum_{deg(\s)\leq 2}||\nabla^\e_{\s, \u} u||_{L^{\infty}(K)}\leq
C_0.$$
Then, for every compact set $K_1\subset\subset K$, for every
$\alpha '<\alpha$ there exists a constant $C>0$ only dependent on
$C_0, \alpha$ and the compact sets, such that
 \begin{equation}
 \sum_{deg(\s) =2} ||\nabla^\e_{\s,
\u} u||_{C^{ \alpha'}_{\u}(K_1)}\leq C.\end{equation}
Moreover,  for every choice of compact sets $K_{2}, K_{3}$ such
that $K_{1} \subset\subset K_{2} \subset\subset K_{3}
\subset\subset  K$ for every function $\phi \in
C_{0}^{\infty}(\text{int}(K))$ such that $\phi\equiv 1$ in
$K_{2}$,  for every multi-index $\s$ of length $2$, we have the
following representation
 \begin{equation}\label{representation1der}
\begin{split}
\nabla^\e_{\s, \bar u}&(u\phi)(x_0)  =      \int\limits_{\O}
\nabla^\e_{\s, \bar u}(x_0)\G^\e_{x_{0}}(\cdot,\x)
     N_{1} (\x,x_0)  d\x
\\&+\int\limits_{\O}  \G^\e_{x_{0}}(\x, 0)
   \nabla^{\e}_{\s, \bar u} (x_0)    P_{x_0}^{1}g(\x)\, \phi(x_0\circ \x^{-1})\,
   d\x
\\&+\int\limits_{\O} \nabla^\e_{\s, \bar
u}(x_0)\G^\e_{x_{0}}(\cdot,\x)\Big(g(\x) - P_{x_0}^{1}g(\x)
\Big)\phi(\x)d\x
\\&+ \sum_{ij=1}^{2n}\int\limits_{\O}  \nabla^\e_{\s, \bar
u}(x_0)\G_{x_{0}}^\e(\cdot,\x) \Big(b_{ij}^\e(\nabla_\u^\e
\u(x_0))- b_{ij}^\e(\nabla_\u^\e \u(\x))\Big)X^\e_{i,\bar u }
X^\e_{j,\bar u }u(\x)\phi(\x)d\x
\\&+\sum_{ijs =1}^{2n}    \int\limits_{\O}
\nabla^\e_{\s, \bar u} X^\e_{i,x_0}(x_0) \G_{x_{0}}^\e(\cdot,\x)
\Big(\bar u(\x)- P^1_{x_0}\bar u(\x)\Big)h_{sij}(x_0)X^\e_{i,\bar
u } X^\e_{j,\bar u }u(\x)\phi(\x)d\x.
\end{split}\end{equation}
where $N_1$ is defined in Proposition \ref{p402}.
\end{proposition}

\begin{proof} By Proposition \ref{p402}, we know that the function
$u\phi$ admits a representation in terms of the fundamental
solution of the frozen operator, and suitable kernels. Let us
verify that these kernels satisfy the assumptions of Lemma
\ref{kerestim} with $k=2$.

From the expression of $N_1$ in (\ref{defN1}), we see that $N_1$
is  a sum of derivatives of the function $\phi$. Since $\phi$ is
constantly equal to 1 on the set $K_2$, then $N_1$ vanishes on the
same set. Hence the point  (i) of Lemma \ref{kerestim} regarding the
support of $N_1$  is satisfied. On the other side $N_1$ depends on
$x_0$ only through the first derivatives of $\u,$ while it depends
on $\x$ through the derivatives up to second order of the function
$u$. Hence it is H\"older continuous in $x_0$ locally uniformly in
$\x$. Hence there exists a constant $C_1$ only depending on $C_0$
such that $$|N_1(\x, x_0) - N_1(\x, x)| \leq C_1 d^\alpha_{\e,
x_0}(x_0, x),$$ for every $x, x_0\in K_1$ and $\x\in K_3$ and this
conclude the proof of assumption (\ref{N1k}).

The second term in (\ref{representation1}) is  the convolution of the
fundamental solution with the function
 $g(\x)=L_{\e, \bar u}u(\x) = P^1_{x_0}g(\x) + \Big(g(\x) -P^1_{x_0}g(\x)\Big).$
The function  $P^1_{x_0}g(\xi) \phi(\x)$ will play the role of the
kernel $N_2$  in Lemma \ref{kerestim}. Since $g$ is of class
$C^{1, \alpha}_\u(K)$, its first order Taylor polynomial is
H\"older continuous in $x_0$ locally uniformly in $\x$ and there
exists a constant $C_1$ only depending on $C_0$ such that
$$|P^1_{x_0}g(\xi) \phi(\x) - P^1_{x}g(\xi) \phi(\x)| \leq C_1
d^\alpha_{\e, x_0}(x_0, x),$$  for every $x, x_0\in K_1$ and
$\x\in K_3$. And this conclude the proof of assumption (ii).

The function $(g -P^1_{x_0}g)(\x)$ satisfies the assumptions
(\ref{N3}) and (\ref{N3bis}) of the kernel $N_{3,k}$, in the same
lemma with $k=2$. Indeed, from the definition (\ref{t301}) of Taylor
polynomial we deduce that
$$
 |(g -P^1_{x_0}g)(\x)|\leq C_1 d^{1+\alpha}_{\e,
x_0}(x_0, \x),$$ if $x,x_0$ are fixed in $K_1$ and  $\x \in K_3$.
Similarly, from   (\ref{tosta}) we deduce that $$
\left|P_{x_{0}}^{1}g(\x) - P_{x}^{1}g(\x)\right| \le C_1 d_{\e,
x_{0}}(x_{0},x)^{\a}
  d_{\e, x_{0}}(x_{0},\x),$$
again with a constant $C_1$, depending on the $C^{1, \alpha}_\u$
norm of $g$. This ensures that (\ref{N3bis}) is satisfied.

In the same way, using the regularity properties of $\u$, we
deduce that the function
$$\Big(b_{ij}^\e(\nabla_\u^\e \u(x_0))- b_{ij}^\e(\nabla_\u^\e
\u(\x))\Big)X^\e_{i,\bar u } X^\e_{j,\bar u }u(\x)$$ satisfies the
assumption  of the kernel $N_{3, k}$ in Lemma \ref{kerestim}.

Finally, from the property (\ref{tosta}) of the Taylor
polynomials, we deduce that the function 

$$
 (\bar u(\x)- P^1_{x_0}\bar
u(\x))h_{sij}(x_0)X^\e_{i,\bar u } X^\e_{j,\bar u }u(\x)\phi(\x)$$
satisfies the assumptions of the kernels $N_{4, ks}$ 
in Remark \ref{osservazione}. 
\end{proof}


In order to obtain an
a-priori estimates of the second derivatives of the solution $u$ in
terms of its $L^p$ norms we need to improve slightly  the previous result.

\begin{lemma}\label {hold1}
Assume that $u$ and  $\bar u$ are $C^{\infty}$ functions,
and that $u$ is a classical solution of $L_{\e, \bar u}
u=g\in C^{\infty}_{\u}(\O)$. Also assume that there exist a
compact $K\subset \O$, and real numbers
 $p>1$, $\alpha<1$ and $C_0>0$ such that
 \begin{equation}\label{stima1}
||\bar u||_{C^{1,\alpha}_{\bar u
 }(K)} +  ||g||_{C^{1,\alpha}_{\bar u }(K)} +$$$$+ ||u||_{C^{1,\alpha}_{\bar u }(K)}+
 \sum_{deg(\s) =2}||\nabla^\e_{\s, \bar u}u||_{L^p(K)}\le C_0. \end{equation}

\begin{itemize}
 \item
  If $Q-p\,\alpha>0$, then, for every compact set $K_1\subset\subset K$
  there exists a constant $C>0$ only depending on $C_0$
such that
\begin{equation}\label{caso1}
\sum_{deg(\s) =2}||\nabla^\e_{\s, \bar u}u||_{L^r(K_1)}\le C,
\end{equation}
where $r=\frac{Q p}{Q-p\,\alpha}$, and $Q$ is the homogeneous
dimension of the space, defined in (\ref{homodimen}).

  \item If $Q-p\,\alpha<0$, then for every compact set $K_1\subset\subset K$
  there exists a constant $C>0$ only depending on $C_0$
such that
\begin{equation}\label{stima23}
\sum_{deg(\s) =2}||\nabla^\e_{\s, \bar u}u||_{L^\infty(K_1)}\le C.
\end{equation}
\end{itemize}
\end{lemma}
\begin{proof} The proof follows from the representation of the
derivatives of $u\phi$ provided in Proposition \ref{esistderivL2}.
%

Let us consider the first integral in (\ref{representation1der})
We first note that, by the expression (\ref{defN1}) of $N_1$,
there exist constants $C_3$ and $C_4$ only dependent on $C_0$ such
that
$$|N_{1}(\x, x_0)|\leq C_3 + C_4\sum_{ij} |X_{i\u}X_{i\u} u|$$
for any $\x, x_0 \in K$.  On the other hand $N_1$ is a sum of
derivatives of the function $\phi$, constant in $K_2,$ the support
of the kernel $N_1$ is a subset of $K_3-K_2$. This implies that if
$x \in K_1$ and $\x\in supp(N_1)$, then $d_{\e, \u}(x,\x) \geq
d_{\e, \u}(K_1, K_3-K_2)$ and  the function
 $\nabla_{\s,\u}^\e(x_0)\G^\e_{x_0} (x,\x)$, is  bounded
 uniformly in $\e$, by condition (\ref{2u}). Then, for every $x_0\in K_1$
$$\Big|\int \nabla_{\s,\bar u}^\e\G^\e (x,\x) N_{1}(\x,
x_0)d\x\Big| \leq \int_{K_3} | N_{1}(\x, x_0)| d\x \leq C\Big( 1+
||X_{i\u}X_{i\u} u||_{L^1(K)}\Big)\leq C_1,
$$ where $C$ only depends on $C_0.$

The second term in (\ref{representation1der}) is  the convolution
of the fundamental solution with a  regular function,
$$\nabla^{\e}_{\s, \bar u} (x_0)\big(   P_{x_0}^{1}g(\xi)\, \phi(x_0\circ
\xi^{-1})\big)$$ whose $L^{\infty}$ norm only depends on
$||g||_{C^{1,\alpha}_{\u}}$  on the support $K$ of the function
$\phi$. 

In the third  term of (\ref{representation1der}), using the
property (\ref{t301}) of the Taylor polynomial, and the estimate
(\ref{2}) of the fundamental solution,  we obtain
$$ \int\limits_{\O} \Big|\nabla^\e_{\s, \bar
u}(x_0)\G^\e_{x_{0}}(\cdot,\x)\Big(g(\xi)-
P_{x_0}^{1}g(\x)\Big)\Big| \phi(\x)d\x\leq C \int\limits_{K_3}
d^{\alpha-Q}_{\e, x_0}(x_0, \x) d\x \leq C_1 $$ for a suitable
constant $C_1$ depending on $||g||_{C^{1,\alpha}_{\u}}$, and on
the compact set $K$.
Consequently,   these  terms belong to $L^\infty_{loc}$.

Using again (\ref{t301}) and (\ref{2})  the last two terms in
representation formula (\ref{representation1der}) can be estimated
by
\begin{equation}\label{granfinale}\int\limits_{\O}d^{-Q + \alpha }_{\e, x_0}(x_0, \x) |X^\e_{i,\bar u } X^\e_{j,\bar u
}u(\x)|d\x. \end{equation}
Thanks to Theorem \ref{fundam}, we can then apply the
standard theory of singular integrals and deduce that if
 $|X^\e_{i,\bar u } X^\e_{j,\bar u }u(\x)|\in L^p(K),$ with
 $Q-p\alpha>0$ then \eqref{caso1} follows.

In order to prove \eqref{stima23} it suffices to apply  H\"older inequality 
to \eqref{granfinale} and use the fact that 
 $Q-p\alpha<0$. This immediately leads  to  the desired $L^{\infty}$ bounds on the second derivatives of the solution.
 \end{proof}

\bigskip

\begin{proposition}\label {Chold2}\label{lineare}
Let us assume that $\u$ is of class $C^{\infty}(\Omega)$ and that
$u$ is a classical solution of
 $L_{\e, \bar u} u=g\in
C^{\infty}_{\u}(\O)$. Let us also assume that there exists a
compact set $K\subset \O$ and constant $C_0$ such that assumption
(\ref{stima1}) is satisfied, and such that
$$||u||_{C^{k-1,\alpha}_\u (K)} + ||\u||_{C^{k-1,\alpha}_\u(K)} +
||g||_{C^{k-2,\alpha}_\u(K)}\leq C_0.$$ Then, for any $2\le k\le
4$, for every compact set $K_1\subset\subset K$ there exists a
constant $C>0$ depending only on the choice of the compact sets,
$C_0, k$ and $\alpha$ such that
$$||u||_{C^{k,\alpha}_{\bar u
 }(K_1)}\le C.$$
\end{proposition}

\begin{proof} A direct computation  shows that the kernel in
representation formula in Proposition \ref{krappres} satisfy
assumptions of Proposition \ref {kerestim}.  
\end{proof}

\bigskip

\subsection{ A priori estimates of the solution of the nonlinear operator
}

We start with the follow iteration result:
\begin{lemma}\label{itera1}
Assume that $z$ is a smooth function   satisfying
\begin{equation}\label{eqz}
 \sum_{ij}a_{ij}(\nabla^\e_{\bar u} \bar u) X^\e_{i\bar u} X^\e_{j \bar u}
z +  f_0 =0 \text{ in } \Omega
\end{equation}
then the function  $v_{h}= X^\e_{h, \bar u} z$
satisfies the equation
$$ \sum_{ij}a_{ij}(\nabla^\e_{\bar u}\u) X^\e_{i\bar u} X^\e_{j
\bar u} v_h   + f_h=0,$$ on the same set $\Omega$, where $f_h$
depends on $\nabla^{\e 2 } _{\bar u}z,$ $\p_{2n}\nabla^\e_{\bar
u}z$, $X^\e_{h, \bar u} f_0$, $\nabla^{\e 2}_{\bar u}\u$.
\end{lemma}
\begin{proof}
Differentiating the equation (\ref{eqz}) with respect to $X^\e_{h, \bar u}$
 we obtain $$0= X^\e_{h, \bar u}
\Big(\sum_{ij}a_{ij}(\nabla^\e_{\bar u}\bar u) X^\e_{i\bar u}
X^\e_{j \bar u} z\Big) + X^\e_{h, \bar u} f_0 =$$ $$= \sum_{ijk}
\p_{p_k}a_{ij} X^\e_{h, \bar u} X^\e_{k \bar u} \bar u X^\e_{i\bar
u} X^\e_{j \bar u} z + \sum_{ij}a_{ij}(\nabla^\e_{\bar u}\bar u)
[X^\e_{h, \bar u} ,X^\e_{i\bar u}] X^\e_{j \bar u} z + $$ $$+
\sum_{ij}a_{ij}(\nabla^\e_{\bar u}\bar u) X^\e_{i\bar u}[X^\e_{h,
\bar u} , X^\e_{j \bar u}]z + \sum_{ij}a_{ij}(\nabla^\e_{\bar
u}\bar u) X^\e_{i\bar u} X^\e_{j \bar u} (X^\e_{h, \bar u} z) +
X_{h, \bar u} f_0.$$
 Then $$
\sum_{ij}a_{ij}(\nabla^\e_{\bar u}u) X^\e_{i\bar u} X^\e_{j \bar
u} v_h +   f_h=0,$$ where
$$f_h=
 \sum_{ij }^{ 2n}
\p_{p_{k}}a_{ij} X^\e_{h, \bar u} X^\e_{k \bar u}\bar u \
X^\e_{i\bar u} X^\e_{j \bar u} z +  $$
$$ \sum_{ij}a_{ij}(\nabla^\e_{\bar u}u) \Big(X^\e_{h, \bar u} b_i -
X^\e_{i, \bar u} b_h\Big) \p_{2n} X^\e_{j \bar u} z + $$ $$+
\sum_{ij}a_{ij}(\nabla^\e_{\bar u}u) X^\e_{i\bar u}\Big(X^\e_{h,
\bar u} b_j - X^\e_{j, \bar u} b_h\Big) \p_{2n}z +$$$$+
\sum_{ij}a_{ij}(\nabla^\e_{\bar u}u) \Big(X^\e_{h, \bar u} b_j -
X^\e_{j, \bar u} b_h\Big) \p_{2n}X^\e_{i\bar u}z $$$$+
\sum_{ij}a_{ij}(\nabla^\e_{\bar u}u) \Big(X^\e_{h, \bar u} b_j -
X^\e_{j, \bar u} b_h\Big) \p_{2n}b_i  \p_{2n}z +  X^\e_{h, \bar u}
f_0.$$
Here the function $f_h$ clearly depends on $\nabla^{\e
2}_{\bar u}z,$ $\p_{2n}\nabla^\e_{\bar u}z$, $X^\e_{h, \bar u}
f_0$,  $\nabla^{\e 2}_{\bar u}\u$.
\end{proof}

\bigskip

\begin{lemma}\label{itera2n}
Assume that $z$ is a smooth function  satisfying
\begin{equation}\label{eqz1}
 \sum_{ij}a_{ij}(\nabla^\e_{\bar u} \bar u) X^\e_{i\bar u} X^\e_{j \bar u}
z +  f_0=0  \text{ in } \Omega \end{equation} then the function
$v=\p_{2n}z,$ satisfies
  $$\sum_{ij}a_{ij}(\nabla^\e_{\bar u}u) X^\e_{i\bar u} X^\e_{j \bar
u} v  + f=0,$$ where $f $ depends on $\nabla^{\e 2}_{\bar u}z,$
$\p_{2n}\nabla^\e_{\bar u}z$ , $\p_{2n}f_0 $.
\end{lemma}

\begin{proof}
Differentiating the equation  (\ref{eqz1}) with respect to $\p_{2n}$
we obtain
$$0 = \p_{2n}\Big(\sum_{ij}a_{ij}(\nabla^\e_{\bar u}\u) X^\e_{i\bar
u} X^\e_{j \bar u} z + f_0\Big)=$$ $$= \sum_{ijk} \p_{p_k}a_{ij}
\p_{2n}X^\e_{k \bar u} \u X^\e_{i\bar u} X^\e_{j \bar u} z +
\sum_{ij}a_{ij}(\nabla^\e_{\bar u}\u) [\p_{2n},X^\e_{i\bar u}]
X^\e_{j \bar u} z + $$ $$+  \sum_{ij}a_{ij}(\nabla^\e_{\bar u}\u)
X^\e_{i\bar u}[\p_{2n}, X^\e_{j \bar u}]z +
\sum_{ij}a_{ij}(\nabla^\e_{\bar u}\u) X^\e_{i\bar u} X^\e_{j \bar
u} (\p_{2n}z) + \p_{2n}f_0.$$ The latter can be rewritten as
$$  \sum_{ij}a_{ij}(\nabla^\e_{\bar u}u) X^\e_{i\bar u} X^\e_{j \bar
u} v + f=0,$$ where  the function $$f= \sum_{ijk} \p_{p_k}a_{ij} \p_{2n}X^\e_{k
\bar u} \u X^\e_{i\bar u} X^\e_{j \bar u} z +
\sum_{ij}a_{ij}(\nabla^\e_{\bar u}\u) \delta_{i 2n-1}\p_{2n} \u
\p_{2n} X^\e_{j \bar u} z + $$ $$+ \sum_{ij}a_{ij}\delta_{j
2n-1}(\nabla^\e_{\bar u}\u) X^\e_{i\bar u}(\p_{2n} \bar u
\p_{2n}z) + \sum_{ij}a_{ij}(\nabla^\e_{\bar u}u) X^\e_{i\bar u}
X^\e_{j \bar u} v + \p_{2n}f_0 $$
 depends on  $\nabla^\e_{\bar u} v,$ $\nabla^{\e 2}_{\bar
u}  z$, $ \p_{2n}\nabla^\e_{\bar u}z ,$ $\nabla^{\e 2}_{\bar u}
\u$.
\end{proof}

In order to
 study of the nonlinear equation
we  apply the previous lemma  with $u=\u$, 
\begin{lemma}\label{iteras}
Let  $\s$  be  a multi-index with all components smaller than $2n$.
Then the function $v_\s=\nabla^\e_{\s, u}u$ satisfies
$$\sum_{ij}a_{ij}(\nabla^\e_{ u}u) X^\e_{i u} X^\e_{j  u}
v_\s + f_\s=0,$$ where $f_\s$ depends on $\nabla^{\e (k+1)}_{ u}
u,$ $\p_{2n}\nabla^{\e k}_{ u}u$ with  $k=deg(\s)$.
\end{lemma}

\begin{proof}
By Lemma \ref {itera1} the assertion  is true for the derivatives of order one.
Assume that it is true for $deg(\s)=k.$
Then, let us consider a multi-index $\s$ of degree  $k+1$.
 By
definition $\s=(\s_1, \tilde \s),$ with $deg(\tilde \s)=k $. Then
by inductive assumption the function $z= \nabla^\e_{\tilde \s,
u}u$ satisfies
 $$\sum_{ij}a_{ij}(\nabla^\e_{ u}u) X^\e_{i u} X^\e_{j u}
z + f_0=0,$$ where $f_0$ depends on $\nabla^{\e k+1 } _u u,$
$\p_{2n}\nabla^{\e k}_{ u}u$. Applying Lemma \ref{itera1} we
deduce that the function $v_\s= X^e_{\s_1 u} z$ is a solution of
$$ \sum_{ij}a_{ij}(\nabla^\e_{  u}u) X^\e_{i\bar u}
X^\e_{j u} v_\s +   f_{\s_1}=0,$$ where $f_{\s_1}$ depends on
$\nabla^{\e 2}_{  u}z,$ $\p_{2n}\nabla^\e_{\u}z$ and $X_{\s_1, u}
f_0$. Since  $f_0$ depends on $\nabla^{\e k+1}_{ u} u,$
$\p_{2n}\nabla^{\e k}_{ u}u$, then $X_{\s_1, u} f_0$ depends on
$X^\e_{\s_1,  u}\nabla^{\e k+1}_{u} u$ and
$$X^\e_{\s_1,
u}\p_{2n}\nabla^{\e k}_{ u}u = \delta_{\s_1 2n-1} \p_{2n} u
\p_{2n}\nabla^{\e k}_{  u}u + \p_{2n}X^\e_{\s_1, u}\nabla^{\e
k}_{ u}u .$$
 \end{proof}
 
 \bigskip

%
%
%
%
%

\begin{theorem}\label{step1}
Let   $u$ be  a smooth  classical solution of the nonlinear
equation $L_{\e} u=0 $. Let us fix a compact set $K\subset\subset
\Omega$ and assume that there exist constants $\alpha <1,$ $p>1$
and  $C_0>0$ such that
\begin{equation}\label{stimausata}  ||u||_{C^{1,\alpha}_{u }(K)}+
 \sum_{deg(\s) =2}||\nabla^\e_{\s, u}u||_{L^p(K)}\le C_0. \end{equation}
 Then, for every $\beta<1$, for every compact set $K_1\subset\subset K$
 there exists a constant $
\tilde C_\beta$ such that
$$||u||_{C^{3,\beta}_{u
 }(K_1)} + ||\p_{2n}u||_{C^{2,\beta}_{ u
 }(K_1)} \le
\tilde C_\beta.$$
\end{theorem}

\begin{proof} We first prove that for every $r>1$ for every compact
set $K_2$ such that $K_1 \subset\subset K_2\subset\subset K$ there
exists a constant $C_r$, only depending on $C_0$ and on the choice
of the compact sets, such that for every multi-index $\s$ of
degree  $2$, we have \begin{equation}\label{Lr}||\nabla^\e_{\s, u}
u||_{L^r(K_2)}\leq C_r.\end{equation} Indeed, for every compact
set $K_3$ such that $K_2 \subset\subset K_3\subset\subset K$ we
can apply the first assertion of Lemma \ref {hold1}, with $u=\u$
and we obtain
\begin{equation} ||\nabla^\e_{\s, u}u||_{L^{r_1}(K_3)}\le C_{r_1},
\end{equation}
where $r_1=\frac{n2 }{n-2\,\alpha}>2$. If $n-r_1\alpha
>0$ we can  apply again Lemma \ref {hold1} on a new compact set
compact set $K_4$ such that $K_2 \subset\subset K_4\subset\subset
K_3$ and we have
\begin{equation} ||\nabla^\e_{\s, \bar u}u||_{L^{r_2}(K_4)}\le C_{r_2},
\end{equation}
with  $$r_2=\frac{nr_1 }{n-r_1\,\alpha}=\frac{2n
}{n-4\,\alpha}>r_1.$$  For every fixed number $r$, after a finite
number of iterations of this same argument, we can prove the
estimate (\ref{Lr}).

Consequently by (\ref{stima23}) we have 
$||\nabla^\e_{\s, \bar u}u||_{L^\infty_{loc}}\le C$ 
then  for every compact set
$K_5$ such that $K_1\subset \subset K_2$  and for every $\beta<1$
there exists a constant $\tilde C_\beta$ such that
$$||\nabla^\e_{\u}u||_{C^\beta _{\u}(K_5)}\leq \tilde C_\beta.$$
 As a consequence of Proposition  \ref{esistderivL2} we
deduce that for every $\beta<1$ for every compact set $K_6$ such
that $K_1\subset \subset K_6\subset \subset K_5$ there exists a
constant $\tilde C_\beta$ such that
$$||u||_{C^{2,\beta}_\u (K_6)}\leq \tilde C_\beta.$$
By Proposition  \ref{Chold2} we deduce that for every $\beta<1$
and for every compact set $K_7$ such that $K_1\subset \subset
K_7\subset \subset K_6$ there exists a constant $\tilde C_\beta$
such that
$$||u||_{C^{3,\beta}_\u(K_7)}\leq \tilde C_\beta.$$
 Applying again the same  proposition  with $k=4,$  we deduce
that
$$||u||_{C^{4,\beta}_\u (K_1)}\leq \tilde C_\beta,$$
which implies in particular that  there exists a constant $\tilde
C_\beta$ such that
$$ ||u||_{C^{3,\beta}_\u(K_1)}  + ||\p_{2n}u||_{C^{2,\beta}_{\bar u
 }(K_1)} \le \tilde C_\beta.$$
\end{proof}

\begin{theorem}\label{ultimo}
Let   $u$ be  a smooth  classical solution of $L_{\e} u=0 $. Let
us also assume that assumption (\ref{stimausata}) is satisfied.
Then, for any compact set $K_1 \subset \subset K,$ for every $
k\in N$ and $\alpha<1$, there exists a constant $C>0$ depending
only on $C_0, k, \alpha$ and $K_1$ such that
$$||u||_{C^{k,\alpha}_{ u
 }(K_1)}\le C.$$
\end{theorem}

\begin{proof} For every  $k \in N$, for every $\s$ with
$deg(\s)=k,$ and components in $\{1, \ldots, 2n\}$ we prove by
induction that
$$ \nabla^\e_{\s, u}u \in C^{3, \alpha}_u(\Omega),\quad
\p_{2n}\nabla^\e_{\s, u}u \in C^{2, \alpha}_u(\Omega),$$ and that
for every compact set $K_1$ such that $K_1\subset \subset  K$ 
there exists a constant $C>0$ depending only on $C_0, k, \alpha$
such that
$$||\nabla^\e_{\s, u}u||_{C^{3,\alpha}_{u
 }(K_1)} + ||\p_{2n}\nabla^\e_{\s, u}u||_{C^{2,\alpha}_{ u
 }(K_1)} \le C.$$
 The thesis is true for
$deg(\s)=0$ by Theorem \ref{step1}.

We assume by induction that
it is true for
 $deg(\s)=k$.
Call $z = \nabla^\e_{\s, u}u$, then  by Lemma \ref{iteras} the
function $z$ satisfies $L_{\e, u}z = f_\s$ in $\Omega$
 with $f_\s=f_\s(\nabla^{\e (k+1)}_{ u}u, \p_{2n}\nabla^{\e k}_{ u}u
)\in
 C^{2, \alpha}_u(\Omega)$, by inductive assumption.
By Lemma \ref{itera2n}, the function $v=\p_{2n}z$ satisfies
$L_{\e, u}v = f,$ in $\Omega$, where the function $f=f(\nabla^{\e
2}_{u}z, \p_{2n}\nabla^\e_{u}z )\in
 C^{1, \alpha}_u(\Omega)$.

It follows  by Proposition  \ref{lineare} that 
$v=\p_{2n}z \in C^{3, \alpha}_u(\Omega)$, and if $K_2$  is a
compact set such that $K_1\subset \subset K_2\subset \subset K$
there exists a constant $C$ depending only on $C_0, k, \alpha$
such that
$$||v||_{C^{3,\alpha}_{u
 }(K_2)}\le C_1 ||f||_{C^{1,\alpha}_{u
 }(K_2)}= C_2.$$
This argument, applied to any multi-index $\s$ with $deg(\s)=k$,
implies that $\p_{2n}\nabla^{\e k}_{ u}u \in C^{3, \alpha}_u.$
Consequently  $\p_{2n}\nabla^{\e k+1}_{ u}u \in C^{2, \alpha}_u,$
and
$$||\p_{2n}\nabla^{\e k}_{ u}u ||_{C^{2,\alpha}_{u
 }(K_2)}\le C_1,$$
 for an other constant $C_1$, only dependent on  $C_0, k, \alpha$.

 Moreover by Lemma \ref{itera1} the function $v_h=X^\e_{h u}z$ satisfies
$Lv_h = f_h,$ with $f_h=f_h(\nabla^{\e 2}_{u}z,
\p_{2n}\nabla^\e_{u}z )\in
 C^{1, \alpha}_u(\Omega)$. Again  Proposition \ref{lineare}
   implies
that $X^\e_{h u}\nabla^\e_{\s, u}u \in C^{3, \alpha}(\Omega),$ and
that

$$||\nabla^{\e (k+1)}_{u}u||_{C^{3,\alpha}_{u
 }(K_1)}\le C_2,$$
 for a constant $C_2$ dependent on  $C_0, k, \alpha$. This
 concludes  the proof.
\end{proof}

\bigskip

{\bf Proof of Theorem \ref{maintheorem}

\begin{proof}
Let $(u_\e)_\e $ be a smooth  approximating sequence of $u$ such that
$$L_\e u_\e = 0\quad \text{ in }\Omega.$$
Let $K_1$ be an arbitrary compact set in $\Omega$. Then there
exist compact sets $K$ and $K_2$ such that $$K_1\subset\subset K_2
\subset \subset K.$$ By assumption there exists a positive
constant $C_0$ such that
$$||\nabla_{E}u_\e||_{L^\infty(K)}\leq C_0,$$
for every $\e$. By proposition \ref{Calphaestimate}  
there exists a constant $C_1$ such that
$$||\nabla^\e_{u_\e}u_\e||_{C^{1, \alpha}_{ u_\e}(K_2)}+
||\p_{2n}u_\e||_{L^{\infty}(K_2)}+||\nabla^{\e
2}_{u_\e}u_\e||_{L^2(K_2)}\leq C_1.$$
Then by Theorem \ref{ultimo}
for every  $k$ there is $C_k$ such that
$$||u_\e||_{C^{2k, \alpha}_{ u_\e}(K_1)}\leq C_k.$$
In particular,
$$||u_\e||_{C^{k}_{ E}(K_1)}\leq C_k.$$
 Since all the constants are independent of $\e$,
letting $\e$ go to  $0$ we obtain estimates of  $u$ in $C^{k}_E$
for every $k$. Consequently $u\in C^{\infty}_E.$
\end{proof}

\end{document}